\documentclass[11pt,a4paper,reqno]{amsart}

\usepackage{amssymb}
\usepackage{latexsym}
\usepackage{amsmath}
\usepackage{graphicx}
\usepackage{amsthm}
\usepackage{empheq}
\usepackage{bm}
\usepackage{booktabs}
\usepackage[dvipsnames]{xcolor}
\usepackage{pagecolor}
\usepackage{subcaption}
\usepackage{tikz,lipsum,lmodern}
\usepackage{enumitem}

\usepackage[margin=2.6cm]{geometry}

\numberwithin{equation}{section}

\theoremstyle{definition}
\newtheorem{theorem}{Theorem}[section]
\newtheorem{corollary}[theorem]{Corollary}
\newtheorem{proposition}[theorem]{Proposition}
\newtheorem{definition}[theorem]{Definition}
\newtheorem{example}[theorem]{Example}
\newtheorem{notation}[theorem]{Notation}
\newtheorem{remark}[theorem]{Remark}
\newtheorem{lemma}[theorem]{Lemma}

\newtheorem{problem}{Problem}

\newcommand\qbin[3]{\textnormal{bin}_{#3}(#1,#2)}

\newcommand{\numberset}{\mathbb}
\newcommand{\N}{\numberset{N}}

\newcommand{\R}{\numberset{R}}
\newcommand{\C}{\mathcal{C}}

\newcommand{\F}{\numberset{F}}

\newcommand{\cC}{\mathcal{C}}
\newcommand{\cL}{\mathcal{L}}
\newcommand{\fq}{\F_q}
\newcommand{\fqn}{\F_{q^n}}

\newcommand{\id}{\textnormal{id}}

\newcommand{\adj}{\textnormal{adj}}

\newcommand{\Nl}{\textnormal{N}_{\textnormal{l}}}
\newcommand{\Nm}{\textnormal{N}_{\textnormal{m}}}
\newcommand{\Nr}{\textnormal{N}_{\textnormal{r}}}

\newcommand{\Il}{\textnormal{I}_{\textnormal{l}}}

\newcommand{\Ir}{\textnormal{I}_{\textnormal{r}}}

\newcommand{\mL}{\mathcal{L}}

\newcommand{\mC}{\mathcal{C}}

\newcommand{\mG}{\mathcal{G}}
\newcommand{\mD}{\mathcal{D}}

\newcommand{\rk}{\textnormal{rk}}

\def\Aut{\mathrm{Aut}}
\newcommand{\mat}{\F_q^{n \times m}}
\renewcommand{\longrightarrow}{\to}

\newcommand{\aut}{\textnormal{Aut}}

\newcommand{\wH}{\omega^{\textnormal{H}}}
\newcommand{\drk}{d^{\textnormal{rk}}}

\newcommand{\GL}{\textnormal{GL}_n(q)}

\newcommand{\Alt}{{\textnormal{Alt}}}
\newcommand{\Sym}{{\textnormal{Sym}}}
\newcommand{\Her}{{\textnormal{Her}}}

\newcommand*{\myproofname}{Proof of the claim}

\usepackage{enumitem}
\setitemize{itemsep=0em}
\setenumerate{itemsep=0em}

\usepackage{hyperref}
\hypersetup{
  colorlinks   = true, 
  urlcolor     = blue, 
  linkcolor    = purple, 
  citecolor   = ForestGreen 
}

\title[]{\textbf{Rank-Metric Codes, Semifields, and the \\ Average Critical Problem}}
\usepackage{authblk}

\usepackage[foot]{amsaddr}
\author{Anina Gruica$^1$}
\author{Alberto Ravagnani$^1$}
\address{$^1$Eindhoven University of Technology, the Netherlands.}
\thanks{$^1$A. G. is supported by the Dutch Research Council through grant OCENW.KLEIN.539. A. R. is supported by the Dutch Research Council through grants OCENW.KLEIN.539 and VI.Vidi.203.045.}
\author{John Sheekey$^2$}
\address{$^2$University College Dublin, Belfield, Dublin, Ireland.}

\author{Ferdinando Zullo$^3$}
\address{$^3$Universit\`a degli Studi della Campania ``Luigi Vanvitelli'', Caserta, Italy. }
\thanks{$^3$F. Z. is very grateful for the hospitality of  Eindhoven University of Technology, the Netherlands, where he was a visiting researcher for two weeks during the development of this research with the support of the DIAMANT Mathematics Cluster, the Netherlands. The research of F. Z. was supported by the project ``VALERE: VAnviteLli pEr la RicErca" of the University of Campania ``Luigi Vanvitelli'' and was partially supported by the Italian National Group for Algebraic and Geometric Structures and their Applications (GNSAGA - INdAM)}

\usepackage{setspace}
\begin{document}

	\maketitle
	
	\thispagestyle{empty}
	
\begin{abstract}
We investigate two fundamental questions intersecting coding theory and combinatorial geometry, with emphasis on their connections. These are the problem of computing the asymptotic density of MRD codes in the rank metric, and the Critical Problem for combinatorial geometries by Crapo and Rota. Using methods from semifield theory, 
we derive two lower bounds for the density function of full-rank, square MRD codes.
The first bound is sharp when
the matrix size is a prime number and the underlying field
is sufficiently large, while
the second bound applies to the binary field.
We then take a new look at the 
Critical Problem for combinatorial geometries, approaching it from a qualitative, often asymptotic, viewpoint. We illustrate the connection between 
this very classical problem and that of computing the asymptotic density of MRD codes.
Finally, we study the asymptotic density of some special families of codes in the rank metric, including the symmetric, alternating and Hermitian ones. In particular, we show that the optimal codes in these three contexts are sparse.
\end{abstract}

\medskip

\section*{Introduction}
This paper focuses on 
two fundamental open problems in coding theory and combinatorial geometry, namely: \ 
(1) Computing the asymptotic density of MRD codes in the rank-metric; and 
(2) Solving new instances of the Critical Problem for combinatorial geometries, proposed by Crapo and Rota.
As we will illustrate, the former problem can be regarded as an ``asymptotic instance'' of the latter.

A rank-metric code is an $\F_q$-linear subspace of the matrix space $\mat$
in which every non-zero matrix has rank bounded from below by a given integer $d$. We assume $m \ge n$ without loss of generality.
 The best known rank-metric codes are called \textit{Maximum Rank Distance} codes (MRD in short). They have the largest possible dimension
for the given matrix size and minimum distance $d$, that is, $m(n-d+1)$. An open question in contemporary coding theory asks to compute the asymptotic behavior of the proportion of MRD codes within those having dimension $m(n-d+1)$, both as the field size $q$ and the column length $m$ tend to infinity.

It has been shown in~\cite{byrne2020partition,antrobus2019maximal}
that MRD codes are not dense, neither for $q \to +\infty$, nor for $m\to+\infty$ (except for trivial parameter choices).
In words, a uniformly random rank-metric code is MRD with probability strictly smaller than onem when either $q$ or $m$ is large. This fact 
is in sharp contrast with the density property of MDS and MRD codes that are linear over the field extension $\F_{q^m}$, shown in~\cite{byrne2020partition,neri2018genericity}, which is exactly what drew the attention onto general
density questions in coding theory.

In~\cite{gluesing2020sparseness}, the exact proportion of $3 \times 3$ MRD codes of minimum distance 3 has been computed.
More recently, it has been proved in~\cite{gruica2020common}  that MRD codes are very sparse as~$q \to+\infty$ (except for trivial parameters), meaning that a uniformly random rank-metric code having dimension $m(n-d+1)$ is MRD with probability that approaches 0 as $q$ grows.

While it has been shown that MRD codes are generally sparse for $q$ large and not dense for~$m$ large, computing the ``exact'' asymptotic behaviour of their density function is a wide open question. A first, quite natural step towards solving this problem is deriving lower bounds for the density function of MRD codes, which is precisely what the first part of this paper focuses on. We combine classical methods from semifield theory with recent classification results to obtain two lower bounds for the density function
of $n \times n$ full-rank MRD codes.
The first bound is sharp when $n$ is prime and $q$ is sufficiently large, and it 
has~\cite[Theorem 2.4]{gluesing2020sparseness} as a corollary; see Theorem~\ref{thm:mrdcount}. The second bound holds for $q=2$ and depends on the number of prime factors of~$n$; see Theorem~\ref{thm:newL} for a precise statement. 
As an application of the first lower bound,
we prove that the density function
of $n \times n$ full-rank MRD codes 
for $q$ large is, asymptotically,
$$\frac{(n-1)(n-2)}{2n} \, q^{-n^3+3n^2-n}, \quad \mbox{if $n$ is prime.}$$

In the second part of the paper, inspired by the density questions about rank-metric codes, we take a new look at a very classical problem in combinatorial geometry, namely, the Critical Problem proposed by Crapo and Rota in~\cite{crapo1970foundations}. The latter is equivalent to computing the number of $k$-dimensional subspaces of a linear space over $\F_q$ that avoid a given set of projective points.
The exact solution to this problem heavily depends
on the combinatorics of the underlying projective points, in a precise lattice-theoretic sense; see, for instance, \cite{crapo1970foundations,hwdl,kung1996critical}.

The question of estimating 
the density function of MRD codes can be regarded as an asymptotic instance of the Critical Problem, since the MRD codes of dimension $m(n-d+1)$  are those subspaces of $\mat$ that avoid the non-zero matrices of rank $d-1$ or less (i.e., the rank-metric ball centered at zero and radius $d-1$, without the center).
This connection 
suggests a different, \textit{qualitative}
way of looking at the Critical Problem itself:
Rather than focusing on its exact solutions, we ask ourselves which \textit{macroscopic} properties of the projective points play a role in determining the number of spaces avoiding them in the \textit{average}. For example, we compute the average number 
of spaces of a given dimension that avoid a set of projective points, as 
the latter ranges over all the sets 
with a given cardinality.
When computing the asymptotics of our results, curious expressions involving the exponential function naturally arise.

We finally compare the density function of MRD codes, when known, with the average solutions for the Critical Problem. This gives us a measure of how ``difficult'' it is for a linear space having the largest allowed dimension to avoid the rank-metric ball. We find that for some problem parameters the rank-metric ball exhibits a quite typical behaviour with respect to avoidance properties. For other parameters the behaviour is on the other hand \textit{very} atypical. We refer to Sections~\ref{sec:avg1} and~\ref{sec:avg2} for a more detailed discussion about this.

The third and last part of the paper studies the density function of rank-metric codes with particular parameters, or made of matrices that obey certain restrictions. We first compute the exact density of 2-dimensional, full-rank $n \times n$ codes, and show that it has the same asymptotics as that of $2 \times m$ full-rank MRD codes. We are able to explain this phenomenon 
with a general result that relates density functions of different rank-metric codes via 3-dimensional tensors; see Theorem~\ref{thm:tens}. We close the paper by investigating the asymptotic density of rank-metric codes made of symmetric, alternating and Hermitian matrices, proving that the optimal ones in these three contexts are always sparse.

\subsection*{Outline} Section~\ref{sec:1} illustrates the two problems studied in this paper, their connection, and collects the preliminaries needed throughout the article. In Section~\ref{sec:semif} we describe the link between MRD codes and semifields, surveying and extending the literature on the topic. Section~\ref{sec:deltan} contains two lower bounds for the density function of MRD codes and discusses their sharpness. In the same section we also briefly illustrate the state of the art on the problem of computing the asymptotic density of MRD codes. A qualitative approach to the Critical Problem for combinatorial geometries is proposed in Sections~\ref{sec:avg1} and~\ref{sec:avg2}, where we also elaborate on its connection with the problem of computing the asymptotic density of MRD codes. In Section~\ref{sec:special} we focus on codes with special parameters, or made of constrained matrices.

\medskip

\section{Rank-Metric Codes and the Critical Problem}
\label{sec:1}
In this section we define rank-metric codes, MRD codes, and we illustrate the connection between these objects and the Critical Problem for combinatorial geometries. We also briefly survey the recent literature on the (wide open) problem of computing the asymptotic density of MRD codes.

\begin{notation}
Throughout this paper, $q$ denotes a prime power and $\F_q$ is the finite field with~$q$ elements. All dimensions are computed over $\F_q$, unless otherwise stated. We also work~with integers $n$, $m$, $k$, $d$ and $N$ that satisfy the following constraints:
$$
m \ge n \ge 2, \qquad
1 \le k \le N-1, \qquad 
2 \le d \le n, \qquad
N \ge 3. 
$$
\end{notation}

\subsection{Rank-Metric and MRD Codes} 
We start by defining rank-metric codes and describing the problem of computing their asymptotic density. The Critical Problem will be discussed later; see Subsection~\ref{sub:crit}.

\begin{definition}
A (\textbf{rank-metric}) \textbf{code} is a non-zero $\F_q$-subspace $\mC \le \mat$. Its \textbf{minimum distance}
is $$\drk(\mC)=\min\{\rk(M) \mid M \in \mC, \; M \neq 0\}.$$ 
\end{definition}

A rank-metric code cannot have large dimension and minimum distance at the same time. The following result by Delsarte expresses a trade-off between these quantities.

\begin{theorem}[Singleton-like Bound; see \cite{delsarte1978bilinear}]
\label{thm:slb}
Let $\mC \le \mat$ be a rank-metric code with $\drk(\mC) \ge d$. We have
\begin{equation*} \label{singletonlikebound}
    \dim(\mC) \le m(n-d+1).
\end{equation*}
\end{theorem}

The best-known rank-metric codes are those having the maximum possible dimension allowed by their minimum distance. 

\begin{definition}
A code $\mC \le \mat$ 
is said to be a \textbf{maximum rank distance}  (\textbf{MRD} in short) code
if it attains the bound 
of Theorem~\ref{thm:slb} with equality, i.e., if it satisfies
$$\dim(\mC) = m(n-\drk(\mC)+1).$$
\end{definition}

Throughout the paper we will extensively use the $q$-binomial coefficient of non-negative integers
$i \ge j$, defined as
\begin{equation}\label{def:qbin}
    \qbin{i}{j}{q}= \prod_{\ell=0}^{j-1}\frac{(q^i-q^\ell)}{\left(q^j-q^{\ell}\right)}.
\end{equation}
It is well known that $\qbin{i}{j}{q}$ counts the number of $j$-dimensional subspaces of an $i$-dimensional space over $\F_q$.

A wide open question in contemporary coding theory asks to compute or estimate the number of rank-metric codes $\mC \le \mat$ having dimension $k$ and minimum distance at least~$d$. Equivalently, it asks to compute or estimate the following quantity.

\begin{notation} \label{not:density}
We let 
$$\delta^\rk_q(n \times m, k, d):=\frac{|\{\mC \le \mat \mid \dim(\mC)=k, \, \drk(\mC) \ge d\}|}{\qbin{mn}{k}{q}}$$ denote the \textbf{density} (\textbf{function}) of rank-metric codes having dimension $k$ and minimum distance at least $d$.
This number measures the proportion of rank-metric codes with minimum distance at least~$d$ within the $k$-dimensional ones. Note that, by definition, the density function of MRD codes is given by
$$\delta^\rk_q(n \times m, m(n-d+1),d).$$
\end{notation}

We can now state the first of the problems 
studied in this paper. When writing ``$q\to+\infty$'' we always consider the limit over the set of prime powers, unless differently stated.

\begin{problem}\label{pb:A}
Compute or estimate $\delta^\rk_q(n \times m, k, d)$. In particular, compute its asymptotic behaviour as 
$q \to +\infty$.
Compute or estimate the asymptotic density of MRD codes, i.e.,
$\delta^\rk_q(n \times m, m(n-d+1),d)$,
as $q \to +\infty$ and as $m \to +\infty$.
\end{problem}

In the past few years, Problem~\ref{pb:A} was mainly studied in the case where the field size $q$ tends to infinity; see \cite{antrobus2019maximal,gluesing2020sparseness,byrne2020partition,gruica2020common}. The case where
$m \to +\infty$ and~$n$ and~$k$ are functions of~$m$
is also very interesting. For example, if $(q,n,d)$ are fixed and $k=m(n-d+1)$, then
$\delta^\rk_q(n \times m, m(n-d+1),d)$ for $m$ large measures the asymptotic proportion of MRD codes as their column length goes to infinity. This quantity is also studied in~\cite{antrobus2019maximal,byrne2020partition,gruica2020common}, although the problem of computing $\lim_{m \to +\infty}\delta^\rk_q(n \times m, m(n-d+1),d)$ is to date open, except for very few choices of the parameters; see~\cite{antrobus2019maximal}.

\begin{notation}
Throughout the paper we will use the highly standard Bachmann-Landau notation (``Big~O'', ``Little~O'', ``$\sim$'' and ``Omega'') to express the asymptotic grows of real-valued functions. 
\end{notation}

The following few results summarize the state of the art on Problem~\ref{pb:A},
to our best knowledge.
In the next statement, we denote by $s_q(m)$ the number of $m \times m$ \textbf{spectrum-free} matrices over~$\F_q$, i.e.,
\begin{align} \label{eq:spect}
    s_q(m) := |\{M \in \F_q^{m \times m} \mid \mbox{ no element of $\F_q$ is an eigenvalue of $M$}\}|.
\end{align}

The number of $2 \times m$ MRD codes of minimum distance $2$ is known and was computed in~\cite{antrobus2019maximal} with the aid of spectrum-free matrices.

\begin{theorem}[\text{\cite[Corollary VII.5]{antrobus2019maximal}}] \label{hejar}
We have $\delta^\rk_q(2 \times m, m, 2)\cdot \qbin{2m}{m}{q}=s_q(m)$.
Furthermore,
$$\lim_{q \to+\infty} \delta^\rk_q(2 \times m,m,2) = \sum_{i=0}^m \frac{(-1)^i}{i!}, \qquad \lim_{m \to +\infty} \delta^\rk_q(2 \times m,m,2) = \displaystyle\prod_{i=1}^{\infty} \left(1-\frac{1}{q^i}\right)^{q+1}.$$
\end{theorem}

In \cite{gluesing2020sparseness}, the 
$3 \times 3$ full-rank MRD codes were explicitly counted using an argument based on semifields. The final result is the following.

\begin{theorem}[\text{\cite[Theorem 2.4]{gluesing2020sparseness}}] \label{thm:heidesemifield}
We have $$\delta^\rk_q(3 \times 3, 3,3)=\frac{(q-1)\, (q^3-1)\, (q^3-q)^3\, (q^3-q^2)^2\, (q^3-q^2-q-1)}{3\, (q^7-1)\, (q^9-1)\, (q^9-q)}\sim \frac{1}{3}q^{-3} \mbox{ as $q \to+\infty$.}$$
\end{theorem}

In~\cite{gruica2020common}, the authors give an upper bound for the number of MRD codes with given 
parameters. The asymptotic version of the bound of~\cite{gruica2020common} for $q \to +\infty$ reads as follows.

\begin{theorem}[\text{\cite[Theorem 5.9]{gruica2020common}}] \label{thm:mrdasybound} 
We have
$$\delta^\rk_q(n \times m, m(n-d+1),d) \in O \left(q^{-(d-1)(n-d+1)+1} \right) \quad \mbox{as $q\to+\infty$}.
$$
\end{theorem}

The previous results shows that MRD codes are sparse whenever $d \ge 2$ and $n \ge 3$. 
Notice however that it does not give the ``exact'' asymptotic behavior of their density function.

In order to state the next results
we introduce the following notation and estimates.
We will repeatedly use them throughout the paper.

\begin{notation} \label{not:pi}
We let
\[
\pi(q,n) := \prod_{i=1}^n \frac{q^i}{q^i-1}
\]
and define $$\pi(q) = \lim_{n \to +\infty}\prod_{i=1}^n \frac{q^i}{q^i-1} = \prod_{i=1}^{\infty} \frac{q^i}{q^i-1}.$$
Notice that $\pi$ is closely related to the Euler function $\phi$, which is defined by $\phi(x)=\prod_{i=1}^\infty (1-x^i)$ for $x \in (-1,1)$. Indeed, we have $\pi(q)=1/\phi(1/q)$.
In the sequel we will also often use the asymptotic estimate of the $q$-binomial coefficient $\qbin{ni}{nj}{q}$ as $n$ tends to infinity: For all integers $i\ge j>0$, 
\begin{equation} \label{eq:pias}
    \qbin{ni}{nj}{q} \sim \pi(q) \, q^{n^2j(i-j)} \quad \textnormal{as $n\to +\infty$}.
\end{equation}
Moreover, we have 
\begin{equation} \label{eq:pias2}
    \qbin{n}{i}{q} \sim \pi(q, \min\{n,n-i\}) \, q^{i(n-i)} \quad \textnormal{as $n\to +\infty$}.
\end{equation}
\end{notation}

To our best knowledge, it is not known whether or not MRD codes are sparse for $m \to+\infty$. The current best bounds are the following. 

\begin{theorem}[\text{\cite[Theorem VII.6]{antrobus2019maximal} and \cite[Theorem 6.6]{gruica2020common}}] \label{thm:minfty} Let $\pi(q)$ be defined as in Notation~\ref{not:pi}. We have
$$\limsup_{m \to +\infty} \, \delta^\rk_q(n \times m, m(n-d+1),d) \le \min\left\{ \frac{1}{\pi(q)^{q(d-1)(n-d+1)+1}}, \frac{1}{\qbin{n}{d-1}{q}\left(\pi(q)-1\right)+1} \right\}.$$
\end{theorem}
It was shown in~\cite{gruica2020common} that the bounds of
\cite{antrobus2019maximal} and~\cite{gruica2020common} 
are not comparable in general. In particular, the minimum in Theorem~\ref{thm:minfty} is not always attained by the same expression.

\subsection{The Critical Problem} \label{sub:crit} In the second part of this paper, we will study the connections between the question of determining the density function of rank-metric codes 
and the Critical Problem for combinatorial geometries. In this subsection we concisely illustrate what this important problem is about. Throughout the paper, we let
$$\mG_q(X,k):=\{V \le X \mid \dim(V)=k\}$$
denote the set of $k$-dimensional subspaces of $X$, called the \textbf{Grassmannian}. Its cardinality is~$\qbin{N}{k}{q}$; see the formula in~\eqref{def:qbin}.

\begin{definition}
A \textbf{point set} in $X$
is a non-empty subset 
$P \subseteq \mG_q(X,1)$.
We say that a subspace $V \le X$
\textbf{distinguishes} (or \textbf{avoids}) $P$ if no element of $P$ is a subspace of $V$.
\end{definition}

The following is a fundamental problem in discrete mathematics. It was proposed by Crapo and Rota in 1970 and it is known as the Critical Problem for combinatorial geometries. The problem admits several  formulations for other combinatorial structures; see \cite{kung1996critical} for an overview. 

\begin{problem}[The Critical Problem; see \text{\cite[Chapter 16]{crapo1970foundations}}] \label{pb:B}
Let $P \subseteq \mG_q(X,1)$ be a point set.
Count the number of $k$-dimensional subspaces of $X$ that distinguish~$P$. In particular, find the largest dimension~$k$ for which such a space exists.
\end{problem}

In this paper, we find it convenient to work with densities. We propose the following notation and terminology.

\begin{notation} \label{not:P}
We denote by $\delta_q(X,k,P)$ the \textbf{density} of $k$-dimensional subspaces of $X$ that distinguish a point set $P \subseteq \mG_q(X,1)$. In symbols,
$$\delta_q(X,k,P):= \frac{|\{V \in \mG_q(X,k) \mid V \mbox{ distinguishes } P\}|}{\qbin{N}{k}{q}}.$$
\end{notation}

Solving Problem~\ref{pb:B} means computing $\delta_q(X,k,P)$. It is well known that this number depends on the combinatorial structure of $P$. More precisely, for given $(q,X,P)$ computing $\delta_q(X,k,P)$ for all $k$ is \textit{equivalent} to computing the characteristic polynomial of the geometric lattice whose atoms are the elements of $P$; see~\cite[Section 3]{hwdl} for the equivalence and the references therein for the context and related results.

\begin{remark} \label{rmk:inst}
Computing the density function of rank-metric codes is a particular instance of the Critical Problem.
Indeed, for $0 \le r \le n$
let 
$$B^\rk_q(n \times m,r):=\{M \in \mat \mid \rk(M) \le r\}$$
be the (\textbf{rank-metric}) \textbf{ball} of radius $r$ centered at $0$.
The span of the non-zero elements of $B^\rk_q(n \times m,r)$ form a point set in $\smash{\mat}$,
which we denote by $P_q^\rk(n \times m,r)$. Note that, by definition, 
$$\delta^\rk_q(n \times m, k,d)=\delta_q(\mat,k,P^\rk_q(n \times m,d-1)).$$
Recall moreover that the size of the rank-metric ball is given by
\begin{equation} \label{def:ball}
    |B^\rk_q(n \times m,r)|= \sum_{i=0}^{r}\qbin{n}{i}{q} \prod_{j=0}^{i-1}(q^m-q^j).
\end{equation}
In particular,
\begin{equation} \label{asy:ball}
    |P_q^\rk(n\times m,r)| \sim \begin{cases}
q^{r(m+n-r)-1} & \mbox{if $q \to +\infty$}, \\[8pt]
\displaystyle \frac{\qbin{n}{r}{q}}{q-1} \; q^{mr} & \mbox{if $m \to +\infty$.}
\end{cases}
\end{equation}
\end{remark}

The interest in the asymptotic behaviour of $\delta^\rk_q(n \times m, k, d)$ together with Remark~\ref{rmk:inst} suggests a new, \textit{qualitative} way of looking at the Critical Problem for combinatorial geometries: Rather than searching for \textit{exact} formulas, one can look for the properties of 
$P$ that determine the value of $\delta_q(X,k,P)$ (for example, the size of $P$ or the dimension of its span over $\F_q$). This is the approach we will take in Sections \ref{sec:avg1} and~\ref{sec:avg2}.

\medskip

\section{MRD Codes and Semifields}
\label{sec:semif}

In this section we illustrate the connection between square, full-rank MRD codes and semifields. 
In Section~\ref{sec:deltan} we will apply this link to
obtain two 
lower bounds for $\delta^{\rk}_q(n \times n, n,n)$.
While the connection between MRD codes and semifields is not new (see~\cite{de2016algebraic,johnson2007handbook,lavrauw2011finite,gluesing2020sparseness} among others), obtaining lower bounds or closed formulas for
$\delta^{\rk}_q(n \times n, n,n)$ requires extending and modifying various known results. 
We start by defining linearized $q$-polynomials, which are crucial in our approach,
and by recalling some of their properties. We refer the reader to~\cite{wu2013linearized} for the proofs.

\begin{definition}
A (\textbf{linearized}) \textbf{$q$-polynomial} over $\F_{q^n}$ is a polynomial of the form
$$ f:=\sum_{i \ge 0}f_i x^{q^i} \in \F_{q^n}[x].$$
The \textbf{$q$-degree} of $f$ is defined as the largest $i$ with $f_i \neq 0$, with the convention that the zero polynomial has $q$-degree $-\infty$.
\end{definition}

\begin{notation}
\begin{enumerate}
    \item The set of $q$-polynomials 
modulo~$x^{q^n}-x$
is an $\F_q$-algebra equipped with the operations of addition and composition of polynomials and scalar multiplication by elements of $\F_q$.
The composition is well-defined for equivalence classes by $[f] \circ [g]:=[f \circ g]$. We denote this $\F_q$-algebra by~$\cL_{n,q}$. 
\item The elements
of~$\cL_{n,q}$ are in one-to-one correspondence with the $q$-polynomials of $q$-degree upper bounded by~$n-1$. 
Throughout the paper we will abuse notation and denote an element of $\cL_{n,q}$ as its unique representative of $q$-degree at most $n-1$. This choice is compatible with the evaluation map $f \mapsto f(a)$, for a fixed $a \in \F_{q^n}$. Indeed, the evaluation of a $q$-polynomial $f$ at $a \in \F_{q^n}$ only depends on its equivalence class modulo $x^{q^n}-x$.
\end{enumerate}
\end{notation}

\begin{remark} \label{rem:end}
It is well known that, as $\F_q$-algebras,
\begin{equation}\label{eq:isom_sigma} (\cL_{n,q},+,\circ)\cong(\mathrm{End}_{\fq}(\fqn),+,\circ) \cong (\F_q^{n \times n},+,\cdot),
\end{equation} 
where in the 3-tuples we omitted the scalar multiplication by an element of~$\F_q$. 
The first isomorphism in~\eqref{eq:isom_sigma} sends  $\smash{f=\sum_{i=0}^{n-1}f_i x^{q^i} \in \cL_{n,q}}$ 
to the endomorphism of $\fqn$ defined by
$\smash{a \mapsto \sum_{i=0}^{n-1}f_i a^{q^i}}$ for all $a \in \fqn$.  The second isomorphism is obtained by representing a linear map as a matrix with respect to an 
$\F_q$-basis of $\F_{q^n}$. In particular,
under these isomorphisms an invertible matrix $M \in \F_q^{n \times n}$ corresponds to a $q$-polynomial defining an invertible $\F_q$-linear transformation.
We refer the reader to~\cite{wu2013linearized} for further details on this. 
\end{remark}

\begin{notation} \label{not:qlin}
By Remark~\ref{rem:end}, a rank-metric code $\C \le \F_q^{n \times n}$ can be seen as an $\F_q$-linear subspace of $\cL_{n,q}$. 
In this section and in the next one 
we will implicitly use this interpretation of rank-metric codes.
\end{notation}

Interpreting rank-metric codes as 
spaces of $q$-polynomials greatly
facilitates their study, for the purposes of this paper.
This will also allow us
to apply the results of~\cite{biliotti1999collineation,sheekey2016new} in our context.

\begin{notation}
If $f=\sum_{i=0}^{n-1}f_ix^{q^i} \in \cL_{n,q}$ and $\rho \in \aut(\F_q)$, then we let $f^\rho:= \sum_{i=0}^{n-1}\rho(f_i)x^{q^i}$.
Furthermore, if $\mC \le \cL_{n,q}$ is a rank-metric code and $\rho \in \aut(\F_q)$, we let $\mC^\rho:=\{f^\rho \mid f \in \mC\}$.
\end{notation}

We define equivalence of rank-metric codes as $\F_q$-subspaces of $\cL_{n,q}$.

\begin{definition}\label{def:equivrk}
Rank-metric codes $\C_1, \C_2 \le \cL_{n,q}$ are \textbf{equivalent}
if there exist invertible \smash{$q$-polynomials} $f,g \in \cL_{n,q}$ and a field automorphism \smash{$\rho \in \mathrm{Aut}(\fq)$} such that
\[ \C_1=f \circ \C_2^\rho \circ g. \]
The \textbf{automorphism group} of $\mC$ is
\begin{align*}
    \aut(\mC) = \{(f_1,\rho,f_2) \in 
    \GL \times \mathrm{Aut}(\fq) \times \GL
    \mid \C=f_1 \circ \C^\rho \circ f_2\}.
\end{align*}
\end{definition}

Notice that the notion of equivalence considered in this paper is different from that considered in~\cite{de2016algebraic}, where the automorphism $\rho \in \aut(\F_q)$ is assumed to be the identity.

When interpreting rank-metric codes 
as subspaces of $\cL_{n,q}$, the family of full-rank MRD codes is closely related to the notion of a \emph{semifield}. In the next part of this section we 
describe this connection. For convenience of exposition, and without loss of generality, we restrict our attention to semifields whose ground set is $\F_{q^n}$.

\begin{definition} \label{def:semifield}
A \textbf{finite semifield} is a triple $(\F_{q^n},+, \star)$ where~$+$ is the usual addition operation of $\F_{q^n}$, and~$\star$ is a binary operation on $\F_{q^n}$ that satisfies the following properties:
\begin{enumerate}
    \item  $x \star y = 0$ implies $x=0$ or $y=0$, for all $x, y \in \F_{q^n}$;
    \item $x \star (y+z) = x \star y + x \star z$ and $(x+y)\star z= x \star z + y \star z$, for all $x,y,z \in \F_{q^n}$;
    \item the multiplicative identity element $1 \in \F_{q^n}$ satisfies $1 \star x = x \star 1=x$ for all $x \in \F_{q^n}$.
\end{enumerate}
A \textbf{finite presemifield} is a triple $(\F_{q^n},+, \star)$ satisfying all of the above, with the possible exception of $(3)$.
Finally, we say that a finite semifield is of \textbf{dimension $n$ over $\F_q$} if 
$(\F_{q^n},+, \star)$ is a (not necessarily associative) algebra over $\F_q$.
\end{definition}

\begin{notation}
For the remainder of this paper, all semifields $(\F_{q^n},+, \star)$ will be assumed to be of dimension~$n$ over $\F_q$.
\end{notation} 

We will also need the concepts of \textit{isotopic} semifields and of \textit{autotopism group} of a semifield. These are defined as follows.

\begin{definition} \label{def:isotopic}
Semifields $(\F_{q^n},+,\star_1)$ and $(\F_{q^n},+,\star_2)$ are  \textbf{isotopic} if there exist invertible, additive maps $f,g,h: \F_{q^n} \to \F_{q^n}$ such that 
$f(x \star_1 y) = g(x) \star_2 h(y)$
for all $x,y \in \F_{q^n}$.
Semifield isotopy is an equivalence relation, whose classes are called \textbf{isotopy classes}. Finally, 
the \textbf{autotopism group} of a semifield $(\F_{q^n},+,\star)$ is
\begin{multline*}
    \aut((\F_{q^n},+,\star)) = \{(f,g,h) \mid f,g,h: \F_{q^n} \to \F_{q^n} \mbox{ are additive and invertible} \\ \mbox{with } f(x \star y) = g(x) \star h(y) \mbox{ for all $x,y \in \F_{q^n}$}  \}.
\end{multline*}
\end{definition}
Isotopy for presemifields is defined analogously. It is known that every presemifield is isotopic to a semifield via Kaplansky's trick, see e.g.\ \cite{lavrauw2011finite}, and so restricting our attention to semifields in the sequel is not restrictive.

The following result states the connection between semifields and MRD codes (represented as spaces of linearized polynomials, as mentioned in Notation~\ref{not:qlin}). In Section~\ref{sec:deltan} we will need to refer not only to the statement of the next theorem, but also to its proof; see Remark~\ref{keyrem}. However, the proof can be skipped in a first reading.

\begin{theorem} \label{thm:1to1}
Isotopy classes of finite semifields of dimension $n$ over $\F_q$ are in one-to-one correspondence with equivalence classes of full-rank MRD codes  $\mC \le \mathcal{L}_{n,q}$.
\end{theorem}
\begin{proof} Let us denote by $\mathcal{S}_{n,q}$ the set of isotopy classes of finite semifields of dimension $n$ over~$\F_q$ and by $\mathcal{F}_{n,q}$ the set of equivalence classes of full-rank MRD codes in $\mathcal{L}_{n,q}$. We will define 
maps
\[\Phi \colon \mathcal{S}_{n,q} \rightarrow \mathcal{F}_{n,q}, \quad  \Psi \colon \mathcal{F}_{n,q} \rightarrow \mathcal{S}_{n,q}\]
and show that they are the inverses of each other. The proof is overall organized into three steps. Throughout the proof, we will use square brackets to denote equivalence classes.
\begin{enumerate}
\item For a semifield $(\F_{q^n},+,\star)$
of dimension $n$ over $\F_q$, let
$\phi((\F_{q^n},+,\star)):= \{R_y \mid y \in \F_{q^n}\},$
where $R_y(x)=x\star y$ for all $x,y \in \F_{q^n}$.
We first show that $\phi((\F_{q^n},+,\star)) \le \cL_{n,q}$ is a full-rank MRD code.
For this, note that $\star$ defines an $\F_q$-bilinear map $\fqn \times \fqn \to \fqn$ that can be written as 
    \begin{align*}
        x \star y = \sum_{i,j=0}^{n-1}c_{ij}x^{q^i}y^{q^j} \quad \mbox{for $x,y \in \F_{q^n}$,}
    \end{align*}
    where $c_{ij} \in \fqn$. Therefore for $y \in \F_{q^n}$ we have 
    \begin{align} \label{eq:spreadel}
        R_y(x)= x \star y =  \sum_{i=0}^{n-1}\left( \sum_{j=0}^{n-1} c_{ij}y^{q^j}\right)x^{q^i} \in \mL_{n,q},
    \end{align}
    where $x$ is viewed as a variable.
    In particular, $\phi((\F_{q^n},+,\star))$ is a subset of $\cL_{n,q}$, and its $\F_q$-linearity follows from
    the distributive laws. Indeed, for $\lambda_1,\lambda_2 \in \F_q$ and $y_1, y_2 \in \F_{q^n}$ we have 
    \begin{align*}
    \lambda_1 R_{y_1}(x) + \lambda_2 R_{y_2}(x) &= \lambda_1 (x \star y_1) + \lambda_2 (x \star y_2) \\ &= x \star (\lambda_1 y_1 + \lambda_2 y_2) \\ &= R_{\lambda_1 y_1 + \lambda_2 y_2}(x).
    \end{align*}
    Note moreover that $R_y(x) = x \star y = 0$ if and only if $x=0$ or $y=0$,
    which shows that~$R_y$ is invertible for all $y \in \fqn^\times$. We also have  $|\phi((\F_{q^n},+,\star))| = q^n$, 
    from which we conclude that $\phi((\F_{q^n},+,\star))$
    is a full-rank MRD code in $\cL_{n,q}$.
    
    \item For an isotopy class $[(\F_{q^n},+,\star)]$ of a semifield of dimension $n$ over 
    $\F_q$, define $$\Phi([(\F_{q^n},+,\star)]):=[\phi((\F_{q^n},+,\star))].$$
    We show that $\Phi$ is well-defined. Suppose that $(\F_{q^n},+,\star_1)$ and $(\F_{q^n},+,\star_2)$ are isotopic semifields. By definition, there exist invertible, additive maps $f,g,h: \F_{q^n} \to \F_{q^n}$ with 
    \begin{align}\label{eq:condequiv1}
    f(x \star_1 y) = g(x) \star_2 h(y) \quad \mbox{for all $x,y \in \F_{q^n}$.}
    \end{align}
    Let $R^1_y$ and $R^2_y$ be the right multiplications corresponding to $\star_1$ and $\star_2$ respectively,
    and let $\mC_1=\phi(\F_{q^n},+,\star_1)$,
    $\mC_2=\phi(\F_{q^n},+,\star_2)$.
    In terms of linearized polynomials, Equation~\eqref{eq:condequiv1} reads 
    \[ f\circ R^1_y=R^2_{h(y)}\circ g. \]
    Therefore, by \cite[Theorem 7]{lavrauw2011finite} we have that $\mathcal{C}_1$ and $\mC_2$
    are equivalent. 
    
    \item Suppose that $\mC \le \cL_{n,q}$ is a full-rank MRD code containing $x \in \cL_{n,q}$. Note that every equivalence class of full-rank MRD codes contains such a $\mC$. Suppose that there exist different $f,g \in \mC$ with $f(1)=g(1)=y$ for some $y \in \fqn$. Then $f-g \in \mC$ is not invertible, since $(f-g)(1) = 0$, a contradiction. Hence $\{f(1) \mid f \in \mC'\}=\fqn$. Thus we can define a map $L\colon \F_{q^n}\mapsto \mC$, where $L(y)$ is the unique element of $\mC$ such that $L(y)(1)=y$. Then $L$ is an invertible $\fq$-linear map, since $L(\lambda_1 y_1+\lambda_2 y_2)(1) = \lambda_1 y_1+\lambda_2 y_2=(\lambda_1 L(y_1)+\lambda_2L(y_2))(1)$, implying $L(\lambda_1 y_1+\lambda_2 y_2)=\lambda_1 L(y_1)+\lambda_2L(y_2)$ for all $\lambda_i\in \fq$, $y_i\in \fqn$, and $L(y)=0$ if and only if $y=0$.
    
    Now we define $\Psi([\mC])=[(\F_{q^n},+,\star)]$, where $x\star y=L(y)(x)$ for any $x,y \in \F_{q^n}$. We now prove that $\Psi([\mC]) \in \mathcal{S}_{n,q}$. Direct computations show that Properties (2) and (3) of Definition \ref{def:semifield} hold.
    Since for all $x \in \fqn$ we have $x \star 1 = L(1)(x) = x = L(x)(1)= 1 \star x$ it follows that $(\F_{q^n},+,\star)$ is a semifield.
    Again, applying \cite[Theorem 7]{lavrauw2011finite} to two equivalent full-rank MRD codes in $\cL_{n,q}$ containing the polynomial $x$, we obtain that $\Psi([\mC])$ does not depend on $\mC$.  
    \end{enumerate}
    It is clear from the definitions above that $\Psi(\Phi([(\F_{q^n},+,\star)]))=\Psi([\phi((\F_{q^n},+,\star))])=[(\F_{q^n},+,\star)]$, and so $\Phi$ and $\Psi$ are the inverse of each other, completing the proof.
\end{proof}

We conclude this section with some notions and results that will be needed in the proof of Lemma~\ref{3.10}.

\begin{definition}
Let $\mC \le \mathcal{L}_{n,q}$ be a rank-metric code. The
\textbf{left idealizer}, \textbf{right idealizer}, \textbf{centralizer} and
\textbf{center} of $\mC$ are defined as follows:
\begin{alignat*}{4}
\Il(\C) &= \{f \in \mathcal{L}_{n,q} \mid f \circ \C\subseteq \C\}, &  &\qquad \mbox{[left idealizer]} \\
\Ir(\C) &= \{f \in \mathcal{L}_{n,q} \mid \C \circ f \subseteq\C\}, & & \qquad \mbox{[right idealizer]} \\ 
\mathrm{Cent}(\C) &= \{f \in \mathcal{L}_{n,q} \mid  f \circ g=g \circ f \textnormal{ for all } g\in \C\}, & & \qquad \mbox{[centralizer]} \\
Z(\C) &= \Il(\C)\cap \mathrm{Cent}(\C). & & \qquad \mbox{[center]}
\end{alignat*}
\end{definition}
The following inclusions are easy to check:
\begin{align*}
    \Il(\C) &\supseteq \{f \in\GL \mid (f,\id,\id)\in \Aut(\C)\}\cup \{0\}, \\
    \Ir(\C) &\supseteq \{f \in\GL \mid (\id,\id,f)\in \Aut(\C)\}\cup \{0\}, \\
\mathrm{Cent}(\C) &\supseteq \{f \in \GL \mid (f,\id,f^{-1})\in \Aut(\C)\}\cup\{0\}.
\end{align*}

In the case where $\C$ contains the $q$-polynomial $x$ and all non-zero elements of $\C$ are invertible, all the inclusions above become equalities; see~\cite[page 440]{sheekey2020new}.

\begin{remark}
Idealizers were originally introduced in \cite{liebhold2016automorphism} and 
represent a useful tool to distinguish between inequivalent rank-metric codes, since they are invariants of an equivalence class.
They have been investigated also in \cite{lunardon2017kernels} under the name of middle and right nuclei, which may be seen as a generalization of the nuclei of a semifield.
As proved in \cite[Proposition 4]{sheekey2020new}, the centralizer and the center are invariants as well.
\end{remark}

\begin{definition}
The \textbf{left-}, \textbf{middle-}, and \textbf{right-nuclei} of a semifield $(\fqn,+,\star)$ are the subsets defined as follows: 
\begin{align*}
\Nl((\fqn,+,\star))&:=\{x \in \fqn ~|~ x \star (y\star z)=(x\star y)\star z, \mbox{ for all } y,z \in {\fqn}\}, &  &\quad \mbox{[left nucleus]} \\
\Nm((\fqn,+,\star))&:=\{y \in {\fqn} ~|~  x \star (y\star z)=(x\star y)\star z, \mbox{ for all } x,z \in {\fqn}\}, &  &\quad \mbox{[middle nucleus]}\\
\Nr((\fqn,+,\star))&:=\{z \in {\fqn} ~|~  x \star (y\star z)=(x\star y)\star z, \mbox{ for all } x,y \in {\fqn}\}. &  &\quad \mbox{[right nucleus]}
\end{align*}
The {\bf nucleus} $\mathrm{N}((\fqn,+,\star))$ of $(\fqn,+,\star)$ is the intersection of the three sets above. Its {\bf center} $Z((\fqn,+,\star))$ is defined by
\[
Z((\fqn,+,\star)) := \{x \in \mathrm{N}((\fqn,+,\star))~|~x\star y = y \star x \mbox{ for all } y\in \fqn\}.
\]
\end{definition}

\begin{remark}\label{rk:nucleiideal}
As shown in~\cite[Proposition 5]{sheekey2020new} (see also~\cite[Theorem 1]{marino2012nuclei}), if $\C \le \cL_{n,q}$ is a full-rank MRD code and $(\fqn,+,\star)$ is its corresponding semifield as in the proof of Theorem~\ref{thm:1to1}, then: 
\begin{itemize}
    \item the left nucleus $\Nl((\fqn,+,\star))$ is isomorphic to $\Il(\C)$,
    \item the middle nucleus $\Nm((\fqn,+,\star))$ is isomorphic to $\Ir(\C)$,
    \item the right nucleus $\Nr((\fqn,+,\star))$ is isomorphic to $\mathrm{Cent}(\C)$.
\end{itemize}
\end{remark}

The proof of Lemma~\ref{3.10}
relies also on
the following concepts. 

\begin{definition}
The \textbf{adjoint} of a $q$-polynomial $f=f_0x+\ldots+f_{n-1}x^{q^{n-1}} \in \cL_{n,q}$ is 
$$f^\adj=\sum_{i=0}^{n-1}f_i^{q^{n-i}}x^{q^{n-i}} \in \cL_{n,q}.$$ The \textbf{adjoint} of a rank-metric code $\mathcal{C} \le \mathcal{L}_{n,q}$ is 
$\mathcal{C}^\adj=\{f^\adj \mid f \in \mathcal{C}\}$.
\end{definition}

Notice that the trace map
$\mathrm{Tr}_{q^n/q}:\F_{q^n} \to \F_q, \, x \mapsto \sum_{i=0}^{n-1} x^{q^i}$
induces a nondegenerate, symmetric $\fq$-bilinear form~$\langle \cdot,\cdot\rangle$ on $\fqn$ via
$$\langle a,b\rangle:=\mathrm{Tr}_{q^n/q}(ab) \mbox{ for $a,b \in \F_{q^n}$}.$$
Moreover, the notions of adjoint and trace map are closely connected with each other. Indeed, for all~$a,b \in \fqn$ we have
\begin{equation}
    \mathrm{Tr}_{q^n/q}(f(a)b)=\mathrm{Tr}_{q^n/q}(a{f}^\adj(b)),
\end{equation}
explaining the choice for the word ``adjoint''. 
A useful identity that we will need later is the following:
\begin{equation}\label{eq:transp}
    (f\circ g)^\adj=(g^\adj \circ f^\adj) \mbox{ for all $f,g \in \mathcal{L}_{n,q}$.}
\end{equation}

\medskip

\section{The Density of MRD Codes via Semifields}
\label{sec:deltan}

In this section we use Theorem~\ref{thm:1to1} and 
the theory of semifields 
to derive two lower bounds
for $\delta^\rk_q(n \times n, n, n)$.
We prove that the first lower bound is sharp when~$n$ is prime and~$q$ is sufficiently large, giving an 
exact formula for the number of full-rank $n \times n$ MRD codes in that case. Our formula generalizes~\cite[Theorem 2.4]{gluesing2020sparseness}; see Theorem~\ref{thm:heidesemifield} and Remark~\ref{rem:hei} below for a more detailed discussion on this.
The second lower bound that we derive only applies when $q=2$ and is obtained by using a different (although always based on semifield theory) argument. Finally, at the end of this section we offer a survey of the state of the art on the problem of computing the asymptotic density of MRD codes, both for $q$ and $m$ large.

This section is organized into three subsections. The first two establish the lower bounds, and the last one surveys the state of the art, taking into account the contributions made by this paper.

\subsection{First Lower Bound}
\label{firstLB}

The goal of this subsection is to establish the following result on the density function of MRD codes.

\begin{theorem}\label{final}
We have
\begin{align} \label{boundm}
    \delta^\rk_q(n\times n, n, n) \ge \frac{|\GL|^2}{n(q^n-1)^2 \, \qbin{n^2}{n}{q}} \left(1+\binom{n-1}{2} \frac{(q^n-1)(q-2)}{q-1}\right).
\end{align}
Moreover, equality holds in~\eqref{boundm} for $n=3$ and arbitrary $q$, and for $n$ prime and $q$ sufficiently large with respect to $n$.
\end{theorem}

For $n=3$, the previous statement
is equivalent to~\cite[Theorem 2.4]{gluesing2020sparseness}.
As already mentioned, we will establish Theorem~\ref{final} and its sharpness
by building on the connection between MRD codes and semifields described in Section~\ref{sec:semif}. Our stepping stone is a classical result by Menichetti showing that when $n$ is prime and $q$ is sufficiently large with respect to $n$, every semifield of 
dimension $n$ over $\F_q$ is isotopic to a generalized twisted field.
In the sequel, for a positive divisor
$\ell$ of $n$ we denote by
$$\smash{N_{q^n/q^\ell} :\F_{q^n} \to \F_{q^\ell}}, \, x \mapsto x^{({q^n-1})/({q^{\ell}-1})}$$
the (relative) field norm. We also let $p$ be the characteristic of $\F_q$, and define $h$ via $q=p^h$.
In particular, we have $|\aut(\F_q)|=h$.

\begin{definition} \label{def:gentwis}
A presemifield $(\fqn, +, \star)$ is called a
\textbf{generalized twisted field} if there exist a positive divisor $\ell$ of $n$, $c \in \F_{q^n}$, and $\F_q$-automorphisms $\alpha,\beta$ of $\F_{q^n}$
 with the following properties:
\begin{itemize}
    \item $\mathrm{Fix}(\alpha)\cap\mathrm{Fix}(\beta)=\mathbb{F}_{q^\ell}$,
    \item $\smash{N_{q^n/q^\ell}(c)\ne 1}$,
    \item $x\star y = xy-c\alpha(x) \beta(y)$ for all $x,y \in \F_{q^n}$.
\end{itemize}
\end{definition}

We are now ready to state Menichetti's result.

\begin{theorem}[\text{\cite{menichetti1977kaplansky,menichetti1996n}}] \label{thm:menichetti}
If $n$ is prime, then there exists an integer $\nu$ (which only depends on~$n$) such that for each $q\geq \nu$ every semifield of dimension $n$ over $\F_q$
is isotopic to a generalized twisted field. 
\end{theorem}

In the sequel, we denote by $\nu(n)$ the minimum value of $\nu$ for which Theorem~\ref{thm:menichetti} holds. Note that $\nu(n)$ is not a prime power in general.

\begin{remark}
In \cite{menichetti1996n} it is shown that $\nu(3)=2$.
For general $n$, an upper bound for~$\nu(n)$ was given by
Menichetti himself in \cite[Proposition 17]{menichetti1996n} in terms of the constant involved in the Lang-Weil lower bound for the number of $\fq$-rational points of a hypersurface of degree~$n$ in $\mathrm{PG}(n-1,\overline{\F}_q)$, where $\overline{\F}_q$ is the algebraic closure of $\fq$. 
Recent results on the $\fq$-rational points of hypersurfaces allow us to give a more explicit upper bound for $\nu(n)$. For example,
it follows from~\cite[Theorem 5.4]{cafure2006improved} that $\nu(n)\leq 2n^4+1$ for all $n$.
\end{remark}

Our next move is to introduce a special class of rank-metric codes in $\cL_{n,q}$.
Their relevance for the proof of Theorem~\ref{final} will be made explicit in Remark~\ref{keyrem}.

\begin{notation} \label{notC_}
Let $\ell$, $c$, $\alpha$ and $\beta$ be as in 
Definition~\ref{def:gentwis}.
Viewing $x$ as an indeterminate, 
we define the rank-metric codes
$$\mC_0:=\{xy \mid y \in \F_{q^n}\} \le \cL_{n,q}, \qquad \C_{c,\alpha,\beta}:= \{xy-c\alpha(x) \beta(y) \mid y\in \F_{q^n}\} \subseteq \cL_{n,q}.
$$

\end{notation}
Note that the conditions  $\mathrm{Fix}(\alpha)\cap\mathrm{Fix}(\beta)=\mathbb{F}_{q^\ell}$ and $N_{q^n/q^\ell}(c)\ne 1$ guarantee that all the non-zero elements in $\C_{c,\alpha,\beta}$ are invertible.
Indeed, if $xy-c\alpha(x) \beta(y)$ is a non-zero element in~$\C_{c,\alpha,\beta}$ that is not invertible, then there exists $z \in \fqn^{\times}$ such that 
\[ c=zy\alpha(z)^{-1}\beta(y)^{-1}, \]
which implies that $N_{q^n/q^\ell}(c)=N_{q^n/q^\ell}(z\alpha(z)^{-1})N_{q^n/q^\ell}(y\beta(y)^{-1})= 1$, a contradiction.
Moreover, if $a \in \F_{q}$ and $xy-c\alpha(x) \beta(y) \in \C_{c,\alpha,\beta}$, then
$a(xy-c \alpha(x)\beta(y))=x(ay)-c\alpha(x)\beta(ay)\in \mathcal{C}_{c,\alpha,\beta}$ and hence $\C_{c,\alpha,\beta}$ is an $\fq$-subspace of~$\mathcal{L}_{n,q}$.

The following observation is crucial in our approach.

\begin{remark} \label{keyrem}
If $n$ is prime and $q \ge \nu(n)$, then every full-rank MRD code $\mC \le \cL_{n,q}$ is equivalent to a code of the form $\mC_{c,\alpha,\beta}$. This follows from combining Theorem~\ref{thm:menichetti} with the proof of Theorem~\ref{thm:1to1}.
Clearly, for arbitrary $n$ and $q$, the number of inequivalent codes of the form~$\mC_{c,\alpha,\beta}$ is a lower bound (not necessarily sharp) for the number of inequivalent full-rank MRD codes~\smash{$\mC \le \cL_{n,q}$}.
\end{remark}

We continue our discussion by surveying two preliminary results, namely~\cite[Theorem 4]{sheekey2016new} and~\cite[Theorem~6.1]{biliotti1999collineation}.
These will be applied later.

\begin{lemma}[\cite{sheekey2016new} and \cite{biliotti1999collineation}] \label{MM}The following hold.
\begin{enumerate}
    \item \label{MM:lem:john1} The automorphism group of $\mC_0$ is the set of  3-tuples 
    $\smash{(ax^{q^i}, \rho, bx^{q^{n-i}})}$ with $\smash{a,b \in \F_{q^n}^{\times}}$,
    $0 \le i \le n-1$,
    and $\rho \in \aut(\F_q)$.  In particular, $|\aut(\mC_0)| = hn(q^n-1)^2$.
\item \label{MM:lem:biliotti} If $\mC_{c,\alpha,\beta}$ and $\mC_{c',\alpha',\beta'}$ are equivalent, and not equivalent to~$\mC_0$, then in the $\F_p$-algebra $\cL_{nh,p} \supseteq \cL_{n,q}$
we have 
$$f \circ \mC_{c,\alpha,\beta} \circ g = \mC_{c',\alpha',\beta'}$$ for some $f=ax^{p^s}, \, g=bx^{p^t} \in \cL_{nh,p}$ with $a,b \in \fqn^\times$ and $0 \le s,t \le hn-1$. 
\end{enumerate}
\end{lemma}

As an application of Lemma~\ref{MM}(\ref{MM:lem:biliotti}),
we give necessary conditions
for the equivalence of codes $\mC_{c,\alpha,\beta}$ and $\mC_{c',\alpha',\beta'}$.

\begin{lemma}\label{lem:tec} 
Let $\alpha, \alpha', \beta, \beta':\fqn \longrightarrow \fqn$ be the $\F_q$-automorphisms defined by
\begin{align*}
    \alpha: x \mapsto x^{q^i}=x^{p^{hi}}, \quad \alpha': x \mapsto x^{q^k}=x^{p^{hk}}, \quad
    \beta: x \mapsto x^{q^j}=x^{p^{hj}}, \quad \beta': x \mapsto x^{q^m}=x^{p^{hm}},
\end{align*}
with $0\leq i,j \leq n-1$ and $\alpha$ different from the identity map.
Suppose that $\mC_{c,\alpha,\beta}$ and $\mC_{c',\alpha',\beta'}$ are equivalent via $\smash{f=ax^{p^s}}$ and $\smash{g=bx^{p^t}}$,
as in the statement of
Lemma~\ref{MM}(\ref{MM:lem:biliotti}). Then one of the following occurs:
\begin{enumerate}
    \item \label{lem:tec1} $s+t=0$, $i=k$ and $c'=(a^{1-p^{hj}}b^{p^{hi+s}-p^{hj+s}})c^{p^s}$;
    \item \label{lem:tec2} $s+t=hk$, $i+k\equiv 0 \pmod{n}$ and $c'= (a^{1-p^{-hj}}b^{p^s-p^{h(i-j)+s}})c^{-p^{-hj+s}}$.
\end{enumerate}
\end{lemma}
\begin{proof}
By assumption, we have $f \circ \C_{c,\alpha,\beta} \circ g = \C_{c',\alpha',\beta'}$. Then for every $y\in \F_{q^n}$ there exists a unique $z\in \F_{q^n}$ such that
\[
ax^{p^s}\circ (xy-c\alpha(x) \beta(y))\circ bx^{p^t} = xz-c'\alpha'(x) \beta'(z),
\]
where the latter is an identity in $\cL_{n,q}$.
This can be re-written as
\[
(ay^{p^s}b^{p^s})x^{p^{s+t}} - (ac^{p^s}y^{p^{hj+s}}b^{p^{hi+s}})x^{p^{hi+s+t}}=zx-(c'z^{p^{hm}})x^{p^{hk}}. 
\]
Since by assumption $hi\ne 0$, we have either (1) $s+t=0$ and $i=k$, or (2) $s+t=hk$ and $hi+s+t=h(i+k)=0$ (modulo $hn$), hence $i+k=0$ modulo $n$.
\begin{enumerate}
    \item Suppose $s+t=0$ and $i=k$. Then $z=ay^{p^s}b^{p^s}$ and $c'z^{p^{hm}}= ac^{p^s}y^{p^{hj+s}}b^{p^{hi+s}}$. Combining these we get that
\begin{align*}
    c'(ay^{p^s}b^{p^s})^{p^{hm}} &= ac^{p^s}y^{p^{hj+s}}b^{p^{hi+s}} \\
    \Longrightarrow \quad  c'a^{p^{hm}}y^{p^{hm+s}}b^{p^{hm+s}} &= ac^{p^s}y^{p^{hj+s}}b^{p^{hi+s}}
\end{align*}
for all $y\in \F_{q^n}$. Therefore $j=m$ and 
\begin{align*}
c'a^{p^{hj}}b^{p^{s+hj}} &= ac^{p^s}b^{p^{hi+s}} \\
\Longrightarrow \quad  c' &= (a^{1-p^{hj}}b^{p^{hi+s}-p^{hj+s}})c^{p^s}.
\end{align*}

\item Suppose $i+k=0$ and $s+t=hk$. Then $z=- ac^{p^s}y^{p^{hj+s}}b^{p^{hi+s}}$ and $c'z^{p^{hm}}= -ay^{p^s}b^{p^s}$. Putting these together we get that
\begin{align*}
c'(ac^{p^s}y^{p^{hj+s}}b^{p^{hi+s}})^{p^{hm}}&= ay^{p^s}b^{p^s}
\\
\Longrightarrow \quad c'a^{p^{hm}}c^{p^{hm+s}}y^{p^{h(j+m)+s}}b^{p^{h(i+m)+s}}&= ay^{p^s}b^{p^s}
\end{align*}
for all $y\in \F_{q^n}$. Thus $h(j+m)=0$, from which $j+m=0$ modulo $n$. It follows that
\begin{align*}
    c'a^{p^{-hj}}b^{p^{h(i-j)+s}}c^{p^{-hj+s}}&= ab^{p^s}
\\
\Longrightarrow \quad c'&= (a^{1-p^{-hj}}b^{p^s-p^{h(i-j)+s}})c^{-p^{-hj+s}}.
\end{align*}
\end{enumerate}
This concludes the proof.
\end{proof}

By applying Lemma \ref{lem:tec}(\ref{lem:tec1}) we obtain the following result, which investigates the equivalence of codes of the form $\C_{c,\alpha,\beta}$ and $\C_{c',\alpha,\beta}$.

\begin{lemma} \label{lem:lem1} 
Let $\smash{\alpha, \beta :\F_{q^n} \longrightarrow \fqn}$ be the automorphisms defined by $\smash{\alpha: x \mapsto x^{q^{i}}}$ and $\smash{\beta:x  \mapsto x^{q^{j}}}$ with $0\leq i,j\leq n-1$. 
The rank-metric code $\C_{c,\alpha,\beta}$ is equivalent to $\C_{c',\alpha,\beta}$ if and only if
\begin{equation}\label{eq:condequiv}
c'=(a^{1-q^{j}}(b^{p^s q^{j}})^{q^{i-j}-1})c^{p^s}
\end{equation}
for some $a,b,s$. Moreover, if $n$ is prime then $a,b,s$ with this property exist if and only if $\smash{N_{q^n/q}(c^{p^s})=N_{q^n/q}(c')}$.
For a given $s$,
there are precisely $(q^n-1)/(q-1)$ values of $c'$ for which this occurs.
\end{lemma}
\begin{proof}
Suppose that $\C_{c,\alpha,\beta}$ and $\C_{c',\alpha,\beta}$ are equivalent. Lemma \ref{lem:tec}(\ref{lem:tec1}) implies that there exist $a,b \in \fqn^\times$ and a non-negative integer $s$ such that $\smash{c'=(a^{1-p^{hj}}b^{p^{hi+s}-p^{hj+s}})c^{p^s}}$. This implies~\eqref{eq:condequiv}.
Conversely, if \eqref{eq:condequiv} holds then let $t=nh-s$ and we can apply Lemma~\ref{lem:tec} to prove that~$\smash{\C_{c,\alpha,\beta}}$ and~$\smash{\C_{c',\alpha,\beta}}$ are equivalent via the polynomials $\smash{f(x)=ax^{p^s}}$ and $\smash{g(x)=bx^{p^t}}$.
Moreover, if \eqref{eq:condequiv} holds then
\[ N_{q^n/q}(c')=N_{q^n/q}((b^{p^s q^{j}})^{q^{i-j}-1})N_{q^n/q}(c^{p^s})=N_{q^n/q}(c^{p^s}). \]
Furthermore, if $\smash{N_{q^n/q}(c')=N_{q^n/q}(c^{p^s})}$ then there exists an element $\eta \in \fqn^{\times}$ such that
\[ c'=\eta^{q-1}c^{q^s}. \]
For $a \in \fqn^{\times}$ one can always find an element $b \in \fqn^{\times}$ with
\[ b^{p^sq^j(q^{i-j}-1)}=\eta^{q-1}a^{q^j-1}. \] 
Thus
\[ c'=(a^{1-q^{j}}(b^{p^s q^{j}})^{q^{i-j}-1})c^{p^s}, \]
concluding the proof.
\end{proof}

The next step is to characterize the codes of the form $\mC_{c,\alpha,\beta}$ that are equivalent to $\mC_0$. This is done in the next preliminary result.

\begin{lemma} \label{3.10} The code $\C_{c,\alpha,\beta}$ is equivalent to $\C_0$ if and only if one between $\alpha$ and $\beta$ is the identity automorphism, or $\alpha=\beta$. 
\end{lemma}
\begin{proof}
Suppose that one of $\alpha$ and $\beta$ is the identity automorphism, or that $\alpha=\beta$.
We will show that
$\mC_{c,\alpha,\beta}$ is equivalent to $\mC_0$ by examining three cases separately.

\noindent \underline{Case 1}. \ If $\alpha=\textnormal{id}$, then
    \begin{align*}
\mC_{c,\alpha,\beta}= \mC_{c,\textnormal{id},\beta} = \{xy-cx\beta(y) \mid y \in \fqn\} = \{x(y-c\beta(y)) \mid y \in \fqn\}.
\end{align*}
We have $N_{q^n / q} (c) \ne 1$, which means that the $\F_q$-linear map $g:\fqn \to \fqn$ to itself defined via $g:y \mapsto y-c \beta(y)$ is bijective. Indeed, if $g(y)=0$ then $y=c \beta(y)$, which contradicts the fact that $N_{q^n / q} (c) \ne 1$. Therefore $g$ is injective (and thus surjective as well). We conclude that~$\C_{c,\textnormal{id},\beta} = \mC_0$.

\smallskip

\noindent \underline{Case 2}. \  If $\beta = \textnormal{id}$, then     \begin{align*}
    \mC_{c,\alpha,\beta}=\mC_{c,\alpha,\textnormal{id}} = \{xy-c\alpha(x)y \mid y \in \fqn\} = \{y(x-c\alpha(x)) \mid y \in \fqn\}.
    \end{align*}
    In analogy to Case 1, 
    the $\F_q$-linear map $f: \fqn \to \fqn$ defined by $f: x \mapsto x-c\alpha(x)$ is bijective. Therefore
    \begin{align*}
        \mC_{c,\alpha,\textnormal{id}} \circ f^{-1} = \mC_0.
    \end{align*}
    
\smallskip

\noindent \underline{Case 3}. \  If $\alpha= \beta$, then 
\begin{align*}
\mC_{c,\alpha,\beta} = \mC_{c,\alpha,\alpha} = \{xy-c\alpha(xy) \mid y \in \fqn\} =\{xy-c(xy)^{q^i} \mid y \in \fqn\}.  
\end{align*}
The adjoint of a polynomial $xy-c(xy)^{q^i} = yx-cy^{q^i}x^{q^i}$ for a given $y \in \fqn$ is $$yx -c^{q^{n-i}}y^{q^{n-i+i}}x^{q^{n-i}} = yx -c^{q^{n-i}}yx^{q^{n-i}} = y(x-c^{q^{n-i}}x^{q^{n-i}}).$$
The adjoint of $\smash{\mC_{c, \alpha, \alpha}}$ is  $\smash{\mC^{\adj}_{c,\alpha,\alpha} = \{y(x-c^{q^{n-i}}x^{q^{n-i}}) \mid y \in \fqn\}}$, which is equivalent to~$\mC_0$ 
because of the same 
argument as in Case 2,
since $\smash{\mathcal{C}_0=\mathcal{C}_{c,\alpha,\alpha}^{\adj}\circ f^{-1}}$. 
Taking the adjoint again, and using~\eqref{eq:transp}, we obtain $\smash{\mathcal{C}_0=(f^{-1})^{\adj} \circ \mathcal{C}_{c,\alpha,\alpha}}$.

To conclude the proof, suppose that neither $\alpha$ nor $\beta$ is the identity automorphism, and that~$\alpha \ne \beta$. In \cite[Theorem 1]{albert1961generalized}, taking into account the connection between nuclei and idealizers of Remark \ref{rk:nucleiideal}, it is shown that the left and right idealizers of $\mathcal{C}_{c,\alpha,\beta}$ are isomorphic to $\mathrm{Fix}(\alpha)$ and $\mathrm{Fix}(\beta)$, respectively. These
cannot be both isomorphic to $\fqn$ (which is the left and right idealizers of $\mathcal{C}_{0}$). Therefore $\mathcal{C}_{c,\alpha,\beta}$ is not equivalent to $\mathcal{C}_0$.
\end{proof}

Another preliminary observation that we will need is the following.

\begin{lemma} \label{lem:restr} 
Let $\alpha, \beta :\F_{q^n} \longrightarrow \fqn$ be defined by $\alpha: x \mapsto x^{q^{i}}=x^{p^{hi}}$ and $\beta:x  \mapsto x^{q^{j}}=x^{p^{hj}}$,
with $1\leq i,j \leq n-1$ and $\alpha\ne\beta$.
There exist $c'\in \fqn^\times$ and an 
$\fq$-automorphism $\beta'$ of $\fqn$
 such that $\smash{\C_{c,\alpha,\beta}}$ is equivalent to $\smash{\C_{c',\alpha^{-1},\beta'}}$.
\end{lemma}
\begin{proof}
We take
$c'=c^{-q^{n-i}}$ and $\beta'=\beta^{-1}$,
which give
$x^{p^{hn-hi}}\circ\C_{c,\alpha,\beta}= \C_{c',\alpha^{-1},\beta'}$.
\end{proof}

By combining Lemmas \ref{lem:lem1}, \ref{3.10} and \ref{lem:restr}
with each other, 
we can compute the number of equivalence classes of codes of the form $\mathcal{C}_{c,\alpha,\beta}$.
By Remark~\ref{keyrem},
if $n$ is prime and $q \ge \nu(n)$,
this quantity coincides with
the number of equivalence classes of full-rank MRD codes $\mC \le \cL_{n,q}$.

\begin{proposition}\label{prop:Purpandmore} 
The number of equivalence classes of codes of the form $\mathcal{C}_{c,\alpha,\beta}$ is 
\begin{equation}\label{eq:genPurp}
1 + (q-2){n-1\choose 2}.
\end{equation}
Moreover, if $n$ is prime then the lower bound is sharp. Furthermore, if $n$ is prime and $q\geq \nu(n)$ then the number of equivalence classes of full-rank MRD codes in $\mathcal{L}_{n,q}$ is given by~\eqref{eq:genPurp}.
\end{proposition}
\begin{proof}
By Lemma \ref{3.10}, if one between $\alpha$
and $\beta$ is the identity map or if $\alpha=\beta$, then $\C_{c,\alpha,\beta}$ is equivalent to $\C_0$.
Note that, by Lemma \ref{lem:lem1}, if $\C_{c,\alpha,\beta}$ and $\C_{c',\alpha,\beta}$ are equivalent then~$c$ and~$c'$ have the same norm over~$\fq$. 
By Lemma~\ref{lem:restr}, we may restrict 
$\alpha$ to an automorphism of the form
$\smash{\alpha:x \to x^{q^i}}$ with $1 \le i \le (n-1)/2$.
In particular, we have 
$(n-1)/2$ choices 
for $\alpha$ that yield to inequivalent codes.
Regarding $\beta$, we have
$n-2$ choices (anything except for the identity and~$\alpha$).
Taking into account also $\cC_0$, we finally obtain the lower bound in the statement.

When $n$ is prime, the condition in~\eqref{eq:condequiv} is equivalent to $N_{q^n/q}(c)=N_{q^n/q}(c')$. This implies that the equivalence classes of the
codes of the form 
$\mathcal{C}_{c,\alpha,\beta}$ are exactly those described in the first part of the proof.
The very last part of the statement follows from Remark~\ref{keyrem}.
\end{proof}

\begin{remark}
The calculation in the proof of  Proposition~\ref{prop:Purpandmore} follows along similar lines to those in \cite{purpura2009counting}. However, the results in \cite{purpura2009counting} are stated and proved only for when $q$ is an odd prime, and therefore we need the above generalization. We also note a small error in \cite{purpura2009counting}, namely the formula incorrectly has ${n-2\choose 2}$ in place of ${n-1\choose 2}$.
\end{remark}

The last step of our argument consists of computing the size of the automorphism group of a code of the form $\mC_{c,\alpha,\beta}$.

\begin{lemma} \label{lem:autsize}
If neither $\alpha$ nor $\beta$ is the identity and $\alpha\ne \beta$, then the size of the automorphism group of $\C_{c,\alpha,\beta}$ is $n(q^n-1)(q-1)\, |\mathrm{Aut}(\F_q)|$. In particular,
if  $\mC_{c,\alpha,\beta}$ is not equivalent to $\C_0$,
then the size of its automorphism group is $n(q^n-1)(q-1)|\,\mathrm{Aut}(\F_q)|$.
\end{lemma}
\begin{proof}
We start by observing that if $\rho \in \mathrm{Aut}(\F_q)$, then $\smash{\C_{c,\alpha,\beta}^\rho=\C_{\rho(c),\alpha,\beta}}$, which coincides with $\smash{\C_{c,\alpha,\beta}}$. Indeed, since $c$ and $\rho(c)$ have the same norm over $\mathrm{Fix}(\beta)$  there exists $\eta \in \fqn^\times$ with $\smash{\rho(c)=\beta(\eta)\eta^{-1}}c$. Therefore
\[ \C_{\rho(c),\alpha,\beta}=\{xy-\beta(\eta)\eta^{-1}c \alpha(x)\beta(y) \mid y \in \fqn\}=\{xy\eta-c \alpha(x)\beta(y\eta) \mid y \in \fqn\}=\C_{c,\alpha,\beta}. \]
Let us determine the elements of the automorphism group of $\C_{c,\alpha,\beta}$ of the form $(f,\mathrm{id},g)$, which amounts to counting the 3-tuples $(a,b,s)$ for which \eqref{eq:condequiv} 
holds and $s\equiv 0\mod h$. The latter condition implies that $c^{p^s-1}$ is a $(q-1)$-th power. Since $\smash{N_{q^n/q}(c)\ne 1}$, there are $q-1$ solutions to the equation $\smash{c^{p^s-1}=w^{q-1}}$ if and only if $\smash{N_{q^n/q}(c^{p^s-1})=1}$. Since $s\equiv 0\mod h$ we have $\smash{N_{q^n/q}(c^{p^s-1})=1}$. Thus for each of the $q^n-1$ choices for $a$, there are $q-1$ choices for $b$. There are $n$ choices for $s$ (namely $0,h,\ldots,(n-1)h$), which completes the proof.
\end{proof}

Finally, by combining
Remark \ref{keyrem},
Lemma \ref{MM},
Proposition~\ref{prop:Purpandmore}, and
Lemma \ref{lem:autsize}, we obtain the main result of this subsection. We state it directly in matrix notation.

\begin{theorem}\label{thm:mrdcount}
The number of $\fq$-linear full-rank MRD codes $\mC \le \F_q^{n\times n}$ is at least
\[
\frac{|\GL|^2}{n(q^n-1)^2}\left(1+\binom{n-1}{2} \frac{(q^n-1)(q-2)}{q-1}\right).
\]
Moreover, the lower bound is attained for $n$ prime and $q \ge \nu(n)$. In particular, it is attained for $n=3$ and any $q$.
\end{theorem}
\begin{proof}
Let $\kappa:={n-1 \choose 2}(q-2)$ and let $\{\C_i=\C_{c_i,\alpha_i,\beta_i} \mid i\in \{0,\ldots,\kappa\}\}$ be a set of representatives for the distinct equivalence classes of codes of the form
$\mC_0$ and $\C_{c,\alpha,\beta}$,
where $\mC_0$ is the same code 
as in Notation~\ref{notC_}. The number of equivalence classes, $\kappa+1$, is given by Proposition~\ref{prop:Purpandmore}.
The number of full-rank MRD codes $\mC \le \cL_{n,q}$ that are 
isomorphic to a code of the form $\mC_i$ for some $i \in \{0,...,\kappa\}$ is
\begin{align}\label{eq:corrnumbfullrank}
\qbin{n^2}{n}{q} \cdot \delta^{\rk}_q(n \times n, n, n) = \displaystyle\sum_{i=0}^\kappa \frac{|\GL|^2 \,  |\Aut(\F_q)|}{|\Aut(\mC_i)|}.
\end{align}
By Lemma \ref{MM}(\ref{MM:lem:john1}), the size of the automorphism group of $\C_0$ is $n(q^n-1)^2h$.
Moreover,
by Lemma~\ref{lem:autsize},
for $i \ge 1$
the size of the automorphism group of $\C_i$ is $n(q^n-1)(q-1)h$.
Substituting the orders of these automorphism groups in \eqref{eq:corrnumbfullrank} and using the fact that $|\aut(\F_q)|=h$, the desired inequality follows. 
The last part of the theorem follows from Remark~\ref{keyrem}.
\end{proof}

\begin{remark}  \label{rem:hei}
Theorem~\ref{thm:mrdcount} coincides with \cite[Theorem 2.4]{gluesing2020sparseness} when $n=3$. Both the results of~\cite{gluesing2020sparseness} and of this section rely on classification results by Menichetti; namely~\cite{menichetti1977kaplansky} for both works, and additionally~\cite{menichetti1996n} in this section. In~\cite{gluesing2020sparseness} the density for the case $n=3$ was found by detailed analysis of the results in~\cite{menichetti1977kaplansky} and~\cite{menichetti1973algebre}; this approach does not seem to extend to larger $n$. In contrast, 
in order to obtain our density results
we utilize the classification results from
~\cite{menichetti1977kaplansky} and~\cite{menichetti1996n}, together with information on the autotopism groups from
~\cite{biliotti1999collineation}.
\end{remark}

Theorem~\ref{final}, which opened this subsection, is now an immediate
consequence of Theorem~\ref{thm:mrdcount}.
By taking the asymptotics as $q \to +\infty$ in 
Theorem~\ref{final} we obtain the following result.

\begin{corollary} \label{cor:exactasy}
We have
\begin{align*}
    \delta^\rk_q(n\times n, n, n)  \in \Omega \left( q^{-n^3+3n^2-n} \right) \quad \textnormal{as $q \to +\infty$.}
\end{align*}   
Moreover, if $n$ is prime we have
\begin{align*}
    \delta^\rk_q(n\times n, n, n)  \sim \frac{(n-1)(n-2)}{2n} \, q^{-n^3+3n^2-n} \quad \textnormal{as $q \to +\infty$.}
\end{align*}
\end{corollary}

Corollary~\ref{cor:exactasy} shows that even though the asymptotic bound on the density function of MRD codes in Theorem~\ref{thm:mrdasybound} gives sparseness, the result is 
not sharp. We will elaborate on this
in Subsection~\ref{subs:state} when updating the state of the art.

\subsection{Second Lower Bound}
In this subsection we present our second lower bound for
the density function of full-rank $n \times n$ MRD codes,
which is obtained using results of Kantor; see~\cite{kantor2003commutative}.
Before stating the result, we discuss how the 
problem of computing $\delta_q^\rk(n \times n, n, n)$
is related to open conjectures in semifield theory.

\begin{remark}\label{rem:kantor}
The problem of computing asymptotic results for \smash{$\delta^\rk_q(n\times n, n, n)$} is relevant to the conjectures made by Kantor in~\cite{kantor2008finite}. Kantor considered the growth characteristics of the number of isotopy classes of semifields of a given order. In particular, he made the following conjectures:
\begin{enumerate}
\item
the number of pairwise non-isotopic semifields of order $q^n$ is not bounded above by a polynomial in $q^n$;
\item
the number of pairwise non-isotopic semifields of order $q^n$ is exponential in $q^n$.
\end{enumerate}
As illustrated in Section~\ref{sec:semif}, semifields and MRD codes with $m=n=d$ are intimately linked. Our consideration of \smash{$\delta^\rk_q(n\times n, n, n)$} differs from these conjectures in the following ways: we consider the absolute number (or, equivalently, the density) of MRD codes rather than the number of equivalence classes; we consider semifields that are $n$-dimensional over $\fq$; and we consider separately the asymptotic behaviour as each of $n$ and $q$ tend to infinity. This setup is natural for MRD codes, and indeed the known results for semifields concern either $n$ or $q$ tending to infinity. 
\end{remark}

As we have seen in Subsection~\ref{firstLB}, knowledge regarding isotopy classes of semifields can be translated into knowledge about the density of MRD codes, provided that sufficient information about automorphism groups is available. In particular, we can translate \cite[Proposition 4.17]{kantor2003commutative} into a lower bound on $\delta^{\rk}_q(n\times n,n,n)$ when~$q=2$ and $n$ is not prime and not a power of $3$. 

\begin{theorem} \label{thm:newL}
Let $\gamma(n)$ denote the number of prime factors of $n$, counted with multiplicities, and suppose $n$ is composite and not a power of $3$. Then we have
\[
\qbin{n^2}{n}{2}\cdot \delta^\rk_2(n\times n,n,n)\geq \frac{|\mathrm{GL}_n(2)|^2  \,  2^n(2^n-1)^{\gamma(n)-2}}{2n}.
\]
\end{theorem}

\begin{proof}
In \cite[Theorem 4.17]{kantor2003commutative}, translated into the language of this paper, it is shown that there exist at least $$\frac{2^n(2^n-1)^{\gamma(n)-2}}{2n}$$
equivalence classes of full-rank MRD codes in $\smash{\mathbb{F}_2^{n\times n}}$ having
trivial automorphism group. Thus, in the same manner as in the proof of Theorem~\ref{thm:mrdcount}, we obtain the claimed formula.
\end{proof}

We note that \cite[Theorem 4.16]{kw2004symplectic} contains a more general construction than~\cite{kantor2003commutative}, but the implications for the asymptotic behaviour of $\delta^\rk_2(n\times n,n,n)$ for $n$ large are similar. In order to obtain lower bounds for $\delta^\rk_q(n\times n,n,n)$, and hence asymptotic lower bounds for $q$ large, we would require a construction of semifields with center containing $\fq$.
However, all of the semifields constructed in \cite{kw2004symplectic} have center $\mathbb{F}_2$.

It is natural to compute the asymptotics of the lower bound of Theorem~\ref{thm:newL} for $n$ large.
We do this by restricting to values of $n$ having the same number of prime factors. The following corollary follows from the asymptotic estimates~\eqref{eq:pias} and~\eqref{eq:gln}.

\begin{corollary} \label{cor:newL}
Let $\gamma \ge 2$ be an integer and  denote by 
$N_\gamma$ the set of integers $n \ge 2$ that are not a power of 3 and with $\gamma(n)=\gamma$. We have
$$\delta^\rk_2(n \times n,n,n) \in \Omega \left(\frac{1}{n} \, 2^{-n^3+3n^2+n( \gamma-1)} \right) \quad \mbox{as $n \to +\infty$, $n \in N_{\gamma}$}.$$
\end{corollary}
Note that the limit in the previous statement makes sense as the set $N_{\gamma} \subseteq \N$ contains infinitely many elements for every $\gamma$.

\subsection{Comparisons and State of the Art} \label{subs:state}
For the convenience of the reader, in this subsection we briefly illustrate the current state of the art on the problem of computing the asymptotic density of MRD codes, in the light of the contributions made by this paper. We do this by listing what the current best known estimates are, both for $q \to+\infty$ and $m \to+\infty$, and by stating which ones are known to be sharp.  This will also give us the chance to compare the new result of this paper with the available literature on the problem.

\begin{enumerate}
    \item The current best upper bound on the density of MRD codes as $q \to +\infty$ is as follows:
\begin{align*}
    \delta^{\rk}_q(n \times m, m(n-d+1),d) \in O\left(q^{-(d-1)(n-d+1)+1}\right) \quad \textnormal{as $q \to +\infty$.}
\end{align*}
For $m=n=d$, this reads
\begin{align*}
    \delta^{\rk}_q(n \times n, n,n) \in O\left(q^{-n+2}\right) \quad \textnormal{as $q \to +\infty$,}
\end{align*}
which can be compared to the lower bound of Corollary~\ref{cor:exactasy}, namely,
\begin{align*}
    \delta^\rk_q(n\times n, n, n)  \in \Omega \left( q^{-n^3+3n^2-n} \right) \quad \textnormal{as $q \to +\infty$.}
\end{align*}  
There is an exponent gap of $n^3-3n^2+2$, implying that the bound of~\cite{gruica2020common} is not sharp.

\item For $m \to +\infty$ and $n$ fixed, the current best asymptotic upper bounds on the density of MRD codes are the ones in Theorem~\ref{thm:minfty}, namely: 
\begin{multline*}
    \qquad \qquad \limsup_{m \to +\infty} \, \delta^\rk_q(n \times m, m(n-d+1),d) \le \\ \min\left\{ \frac{1}{\pi(q)^{q(d-1)(n-d+1)+1}}, \frac{1}{\qbin{n}{d-1}{q}\left(\pi(q)-1\right)+1} \right\},
\end{multline*}
where $\pi(q)$ is defined in Notation~\ref{not:pi}.

\item For the asymptotic upper bound for $\delta_q^\rk(n \times n, n, n)$ as $n \to +\infty$, as in the proof of~\cite[Theorem 6.6]{gruica2020common} one can show that the upper bound on \smash{$\delta_q^{\rk}(n \times n, n, n)$} in~\cite[Theorem 5.7]{gruica2020common} is asymptotically  
\begin{align*}
    \frac{1}{\qbin{n}{n-1}{q}\left(\pi(q)-1\right)+1} \quad \mbox{as $n \to+\infty$,}
\end{align*}
where $\pi(q)$ is defined in Notation~\ref{not:pi}.
This gives $$\delta_q^{\rk}(n \times n, n, n) \in O\left(q^{-n}\right) \quad \mbox{as $n \to +\infty$}$$
and thus $\lim_{n \to +\infty}  \delta_q^{\rk}(n \times n, n, n) = 0$. 
The upper bound in~\cite{antrobus2019maximal} can be translated into
$\smash{\delta_q^{\rk}(n \times n, n, n) \in O\left(\pi(q)^{-qn}\right)}$ as $n \to +\infty$. This bound also gives the sparseness of full-rank MRD codes in $\F_q^{n \times n}$ as $n \to +\infty$. Note that in general the bounds are not comparable, so we get
\begin{align} \label{eq:upperboundn}
     \delta_q^\rk (n \times n, n, n) \in O\left(\min\{q^{-1},\pi(q)^{-q}\}^n\right)\quad \textnormal{as $n \to +\infty$.}
\end{align}
One can check that for $q=2$ we have $\min\{q^{-1},\pi(q)^{-2}\} = \pi(2)^{-2}$ where $\pi(2)^{-2}$ takes the approximate value of 0.0833986; see~\cite[Remark VII.2]{antrobus2019maximal}.
Therefore we can compare this asymptotic upper bound with the lower bound obtained in Corollary~\ref{cor:newL} for $q=2$ and for a fixed value of $\gamma \ge 2$:
\begin{align} \label{eq:upperboundn2}
    \delta^\rk_2(n \times n,n,n) \in \Omega \left(\frac{1}{n} \, 2^{-n^3+3n^2+n( \gamma-1)} \right) \quad \mbox{as $n \to +\infty$, $n \in N_{\gamma}$}
\end{align}
where $N_\gamma$ denotes the set of integers $n \ge 2$ that are not a power of 3 and with $\gamma(n)=\gamma$. From the approximate value of $\pi(2)^{-2}$ it follows that $\pi(2)^{-2n} \ge 2^{-4n}$ and thus the asymptotic upper bound has an exponent gap of at least $n^3-3n^2-n(\gamma+3)$. Therefore for $q=2$ the bounds of~\eqref{eq:upperboundn} and~\eqref{eq:upperboundn2} are far apart. In particular,
the exact asymptotic estimate of $\delta_2^{\rk}(n \times n, n, n)$ as $n \to +\infty$ remains unsettled.
\item As a last item we survey the current sharp asymptotic estimates of the density function of MRD codes both as $q \to +\infty$ and $m \to +\infty$.
In~\cite{antrobus2019maximal} the following asymptotic densities were shown:
    \begin{align*}
        \qquad \lim_{q \to+\infty} \delta^\rk_q(2 \times m,m,2) = \sum_{i=0}^m \frac{(-1)^i}{i!},  \quad \lim_{m \to +\infty} \delta^\rk_q(2 \times m,m,2) = \displaystyle\prod_{i=1}^{\infty} \left(1-\frac{1}{q^i}\right)^{q+1}.
    \end{align*}
The only other exact asymptotic density of MRD codes obtained so far is the one in Corollary~\ref{cor:exactasy}, which is
\begin{align*}
    \delta^\rk_q(n\times n, n, n)  \sim \frac{(n-1)(n-2)}{2n} \, q^{-n^3+3n^2-n} \quad \textnormal{as $q \to +\infty$,}
\end{align*}
for $n$ a prime number. When $n=3$, the latter asymptotic estimate can also be derived from~\cite{gluesing2020sparseness}.
\end{enumerate}

\medskip
\section{The Average Critical Problem and Rank-Metric Codes} \label{sec:avg1}

Finding closed formulas for $\delta_q(X,k,P)$, where $P$ is an arbitrary point set (see Notation~\ref{not:P}), is a difficult task in general. Several open questions in discrete mathematics are instances of this problem, including the celebrated MDS Conjecture by Segre~\cite{segre1955curve,dowling1971codes,ball2020arcs}.
In the next two sections of the paper, as already mentioned in the introduction, we take a more \textit{qualitative}, often \textit{asymptotic}, approach to Problem~\ref{pb:B}. More precisely, we ask ourselves what the \textit{average} value of $\delta_q(X,k,P)$ is, when~$P$ ranges over the collection of point sets
having a particular property. Moreover, 
we study the asymptotic behaviour of this value as some problem parameters tend to infinity.

The purpose of this study is twofold. On the one hand, 
understanding how much the behaviour of the rank-metric ball
deviates from that of a typical point set having the same cardinality with respect to the value of
$\delta_q(X,k,P)$.
On the other hand,
identifying which structural properties of a point set $P$ determine the value of
$\delta_q(X,k,P)$. 
The latter problem is quite natural and will be studied in Section~\ref{sec:avg2}. For clarity of exposition, we illustrate the former problem with an example from recent coding theory literature. 

\begin{example} \label{ex:explain}
Let $\smash{X=\F_q^{2 \times m}}$. Then the asymptotic density 
of the $2 \times m$ MRD codes of minimum distance 2 is given by the curious formulas in Theorem~\ref{hejar}. These rank-metric codes are the
subspaces of $\smash{\F_q^{2 \times m}}$ that distinguish the point set $\smash{P_q^\rk(2 \times m,1)}$; see Remark~\ref{rmk:inst}. The latter has asymptotic size $q^{m}$
as $q \to +\infty$ and 
$q^m(q+1)/(q-1)$ as $m \to+\infty$; see the estimates in~\eqref{asy:ball}.
A natural and yet quite interesting question is whether the rank-metric ball
behaves like a typical point set of the same size with respect to the number of spaces distinguishing it. In other words, what is the average number of $m$-dimensional spaces of $\smash{\F_q^{2 \times m}}$
that distinguish a point set of the same asymptotic size as
$\smash{P_q^\rk(2 \times m, 1)}$?
In Theorem~\ref{thm:alphahat0} we will answer a  general version of this question.
We consider its asymptotic analogue for the matrix space in Theorem~\ref{asydelta}.
\end{example}

In this section we compute the average number of $k$-dimensional spaces distinguishing a point set $P$, when $P$ ranges over all point sets having a certain cardinality.
We then compute the asymptotics of the formulas we obtain as some parameters go to infinity. In Subsection~\ref{sub:mrdavg} we will use these results to compare the behaviour of the rank-metric ball with the average set in~$\smash{\mat}$ having the same cardinality.

\begin{notation}
 For an integer $1 \le \ell \le (q^N-1)/(q-1)$, we let
\[
\hat{\delta}_q(N,k,\ell):= \frac{\displaystyle\sum_{\substack{P \subseteq \mG_q(X,1) \\ |P|=\ell}} \delta_q(X,k,P)}{ \dbinom{\frac{q^N-1}{q-1}}{\ell}}\]
denote the average density of the $k$-dimensional subspaces of $X$ that
distinguish a point set~$P$, as $P$ ranges over all point sets of size $\ell$. Note that the value of $\hat{\delta}_q(N,k,\ell)$ does not depend on the choice of $X$, but only on its dimension. This is reflected in the notation we chose.
\end{notation}

We can give a simple expression for $\hat{\delta}(N,k,\ell)$ as a ratio of binomial coefficients. More in detail,
by counting the elements of the set $$\{(P,V) \mid P\subseteq \mG_q(X,1), \, |P|=\ell, \, V\in \mG_q(X,k), \, V \mbox{ distinguishes } P \}$$
in two ways one obtains
$$\sum_{\substack{P \subseteq \mG_q(X,1) \\ |P|=\ell}} \delta_q(X,k,P) = 
\dbinom{\frac{q^N-q^k}{q-1}}{\ell}.$$
This gives the following result.

\begin{theorem}\label{thm:alphahat0}
For all $1 \le \ell \le (q^N-1)/(q-1)$ we have 
\[
\hat{\delta}_q(N,k,\ell)
 =  \frac
{
\dbinom{\frac{q^N-q^k}{q-1}}{\ell}
}
{
\dbinom{\frac{q^N-1}{q-1}}{\ell}
}.
\]
\end{theorem}

In this section we are also interested in the asymptotic behaviour of $\hat{\delta}_q(N,k,\ell)$ as some of the problem parameters tend to infinity.
Later in the paper we will compare our results with the asymptotic behaviour of certain rank-metric codes. 
We investigate two general scenarios, namely: (1) $N$ and $k$ are fixed, $q$ goes to infinity, and $\ell$ is a fixed power of $q$; (2) 
$q$ is fixed, $k$ goes to infinity, $N$ is a constant multiple of $k$, and $\ell$ grows exponentially in $k$. 
The exponential function will arise when computing the asymptotics of (ordinary) binomial coefficients. We denote it by
$\exp: \R \to \R$.

We start with a preliminary result on the asymptotics of the ratio of binomial coefficients. In the proof use Stirling's well-known approximation for the factorial.

\begin{lemma}
Let $(x_i)_i$, $(y_i)_i$ and $(z_i)_i$ be positive integer sequences with $x_i/y_i \to +\infty$ and 
$x_i/z_i \to +\infty$ as $i \to +\infty$. We have
$$\frac{\dbinom{x_i}{y_i}}{\dbinom{x_i+z_i}{y_i}} \sim \exp\left({-\frac{y_iz_i}{x_i}}\right) \ \mbox{ as $i \to +\infty$}.$$
\end{lemma}
\begin{proof}
Since $x_i/y_i \to +\infty$ as $i \to +\infty$ and $(z_i)_i$ is a positive sequence by assumption, we have $(x_i+z_i)/y_i \to +\infty$ as $i \to +\infty$ as well. We can therefore apply Stirling's approximation as follows:
$$\dbinom{x_i}{y_i} \sim \frac{x_i^{y_i}}{y_i!}
\ \mbox{ as $i \to +\infty$}, \qquad  
\dbinom{x_i+z_i}{y_i} \sim \frac{(x_i+z_i)^{y_i}}{y_i!}
\ \mbox{ as $i \to +\infty$}.$$
Thus
$$\frac{\dbinom{x_i}{y_i}}{\dbinom{x_i+z_i}{y_i}} \sim {\left(1+\frac{z_i}{x_i}\right)}^{-y_i}
= {\left(1+\frac{1}{x_i/z_i} \right)}^{\frac{x_i}{z_i}\frac{z_i}{x_i}(-y_i)}
\sim \exp\left({-\frac{y_iz_i}{x_i}}\right) \ \mbox{as $i \to +\infty$},$$
where the latter asymptotic estimate follows from the fact that $x_i/z_i \to +\infty$ as $i \to +\infty$ by assumption.
\end{proof}

We are now ready to compute the asymptotics of the ratio in Theorem~\ref{thm:alphahat0}
in some scenarios that are particularly relevant for us.

\begin{proposition} \label{pp}
\begin{enumerate}
\item \label{pp:e1}
Let $1 \le s <N-1$ be an integer and let $(\ell_q)_q$ be an integer sequence with 
$\ell_q \sim q^s$ as $q\to +\infty$. We have
$$\hat{\delta}_q(N,k,\ell_q) \sim \exp\left({-q^{k+s-N}}\right) \  \mbox{ as $q\to +\infty$}.$$
\item \label{pp:e2}
Let $1 \le k',r <n$ be integers and let 
$\ell' >0$ be a real number.
Let $(\ell_m)_m$ be an integer sequence and suppose that $\ell_m \sim\ell'q^{mr}$ as $m \to +\infty$. We have
$$\hat{\delta}_q(mn,mk',\ell_m) \sim \exp\left({-\ell' q^{m(k'+r-n)}}\right)\  \mbox{ as $m\to +\infty$}.$$
\end{enumerate}
\end{proposition}

The previous results show that, for some choice of the parameters, the average density of the distinguishing spaces converges to a positive number. The latter is obtained by evaluating the exponential function. We state one of these cases in a corollary.

\begin{corollary}
The average density of hyperplanes of $X$ distinguishing a point set $P$ is $1/e$, if $P$ is a uniformly random point set of cardinality $q$.
\end{corollary}

\subsection{Comparison Between MRD Codes and the Average} \label{sub:mrdavg}
The goal of this short subsection is to compare the ``average'' solution to the Critical Problem with the density function of MRD codes. This analysis indicates how much the behaviour of the rank-metric (point set) ball deviates from the behaviour of a uniformly random point set in $\smash{\mat}$,
with respect to the number of avoiding spaces. We start by specializing 
the results of the previous subsection to the matrix space.

\begin{theorem}\label{asydelta}
We have 
$$\hat{\delta}_q(mn,
m(n-d+1),
|P_q^\rk(n \times m,d-1)|) \sim 
\begin{cases}
\exp\left({-q^{d(n-d+2)-n-2}}\right) & \mbox{as $q \to +\infty$,} \\[8pt]
\exp \left(-\frac{\qbin{n}{d-1}{q}}{q-1}\right) & \mbox{as $m \to +\infty$.}
\end{cases}$$
\end{theorem}
\begin{proof}
The asymptotic size of $P_q^\rk(n \times m,d-1)$ is given in~\eqref{asy:ball}. The estimate in the statement for $q \to +\infty$
follows from Proposition~\ref{pp}(\ref{pp:e1}) and straightforward computations,
taking $N=mn$, 
$\smash{\ell_q=|P_q^\rk(n \times m,d-1)|}$,
and $s=(d-1)(m+n-d+1)-1$.
Analogously, the estimate for $m \to +\infty$ follows from Proposition~\ref{pp}(\ref{pp:e2}) taking
$k'=n-d+1$, $r=d-1$, and $\ell'=\qbin{n}{d-1}{q}/(q-1)$.
\end{proof}

In the remainder of this subsection we compare the previous result to the exact values of the asymptotic density
of MRD codes, when these are available. Each of the following three remarks concentrates on a different parameter set.

\begin{remark}
Take $n=d=2$ and $q$ arbitrary. By Theorem~\ref{asydelta} we have
$$\hat{\delta}_q(2m,m,|P_q^\rk(2 \times m,1)|) \sim e^{-(q+1)/(q-1)} \quad \mbox{as $m \to+\infty$.}$$
This constant should be compared with the asymptotic density of $2\times m$ full-rank MRD codes for $m$ large, which is
$$\displaystyle\prod_{i=1}^{\infty} \left(1-\frac{1}{q^i}\right)^{q+1},$$
as computed in~\cite{antrobus2019maximal} and stated in Theorem~\ref{hejar}.
We have 
\begin{equation} \label{ine}
    \displaystyle\prod_{i=1}^{\infty} \left(1-\frac{1}{q^i}\right)^{q+1} < e^{-(q+1)/(q-1)},
\end{equation}
as we will show shortly. This tells us that, in the limit for $m \to+\infty$, the rank-metric (point set) ball of radius $1$ in $\F_q^{2\times m}$ has fewer distinguishers of dimension $m$ than the average point set of the same cardinality.
In other words, the rank-metric ball is ``harder'' to distinguish than the average point set having the same cardinality. It can be checked that for $q$ large the two sides of~\eqref{ine} are very close.

In order to establish~\eqref{ine}, observe that the desired inequality holds if and only if the inequality obtained by taking the natural logarithm ($\log$) of both sides holds. This follows from the fact that $x \mapsto \log(x)$ is continuous and increasing. Thus~\eqref{ine} is equivalent to
\begin{equation} \label{ine2}
\frac{1}{q-1} < \sum_{i=1}^\infty \log\left( \frac{q^i}{q^i-1}\right).
\end{equation}
Note moreover that $q^i/(q^i-1)=1+1/(q^i-1)$.
Thus using the Taylor series expansion of the natural logarithm around $0$ we find that
$$\log\left( \frac{q^i}{q^i-1}\right) \ge \frac{1}{q^i-1}-\frac{1}{2(q^i-1)^2} \ge 0.$$
It follows that
$$\sum_{i=1}^\infty \log\left( \frac{q^i}{q^i-1}\right) \ge \sum_{i=1}^5 \left( \frac{1}{q^i-1} - \frac{1}{2(q^i-1)^2} \right)>\frac{1}{q-1},$$
where the latter inequality can be shown by applying elementary methods from Calculus.
All of this establishes~\eqref{ine}.
\end{remark}

\begin{remark}
Take $n=d=2$ and $m$ arbitrary. By Theorem~\ref{asydelta} we have
\begin{equation} \label{eee1}
\hat{\delta}_q(2m,m,|P_q^\rk(2 \times m,1)|) \sim 1/e \quad \mbox{as $q \to+\infty$.}
\end{equation}
This limit value should be compared with the asymptotic density of $2 \times m$ MRD codes of distance $2$ for $q$ large, again computed in~\cite{antrobus2019maximal} and reading
\begin{equation} \label{eee2}
\delta^{\rk}_q(2 \times m,m,2) \sim \sum_{i=0}^m \frac{(-1)^i}{i!} \quad \mbox{as $q\to+\infty$};
\end{equation}
see Theorem~\ref{hejar}.
This time the rank-metric (point set) ball exhibits an alternating behaviour. Indeed, we have $\sum_{i=0}^m (-1)^i/i!>1/e$ 
if and only if $m$ is even. In other words,
for $q$ large the rank-metric (point set) ball has more distinguishers than the average set of the same cardinality if 
$m$ is even, and less distinguishers if $m$ is odd.
\end{remark}

\begin{remark}
Take $n=m=d$ prime. In Corollary~\ref{cor:exactasy} we have shown that
\begin{align}\label{fff1}
    \delta^\rk_q(n\times n, n, n)  \sim \frac{(n-1)(n-2)}{2n} \, q^{-n^3+3n^2-n} \quad \textnormal{as $q \to +\infty$.}
\end{align}
We compare this against
\begin{equation} \label{fff2}
    \hat\delta_q(n^2,n,|P_q^\rk(n \times n,n-1)|) \sim 1/e^{q^{n-2}} \quad \textnormal{as $q \to +\infty$.}
\end{equation}
Therefore, for $n$ prime, the rank-metric (point set) ball of radius $n-1$ in $\F_q^{n \times n}$ has significantly more distinguishers than the average set having the same cardinality. Moreover, the densities of the distinguishing spaces follow completely different distributions, since~\eqref{fff1} exhibits a polynomial decay and~\eqref{fff2} an exponential one.
\end{remark}

\medskip
\section{A Qualitative Approach
to the Critical Problem}
\label{sec:avg2}
In this section we explore 
which \textit{macroscopic} properties of a point set $P \subseteq \mG_q(X,1)$ determine the value of~$\delta_q(X,k,P)$. While the exact value depends on the characteristic polynomial of the geometric lattice generated by $P$, in this paper we are mainly interested in understanding the  interdependence between $P$ and $\delta_q(X,k,P)$ from a \textit{qualitative} perspective.

A parameter that seems to play a decisive role is the dimension of the space generated by the sum of all the elements of $P$. Throughout this section we abuse notation and write
$$\langle P \rangle:= \sum_{L \in P} L \le X, \qquad \dim(P):=\dim(\langle P \rangle).$$

We start with two examples showing a somewhat counter-intuitive, positive correlation between $\dim(P)$ and $\delta_q(X,N-1,P)$, for a fixed cardinality $|P|$. We first establish the results and then elaborate on them.

\begin{proposition}
Let $P\subseteq \mG_q(X,1)$ be a point set with $i:=|P| \ge 2$.
\begin{enumerate}
    \item Suppose $\dim(P)=2$ and $i \le q$. Then
    \[
\delta_q(X,N-1,P) = \frac{(q+1-i)(q-1)\,q^{N-2}}{q^N-1}.
\]
\item Suppose $i \le N-1$ and $\dim(P)=i$.
Then \[
\delta_q(X,N-1,P) = \frac{(q-1)^i\,q^{N-i}}{q^N-1}.
\]
\end{enumerate}
\end{proposition}

\begin{proof}
\begin{enumerate}
    \item The number of hyperplanes of $\smash{\F_q^N}$ that meet a given 2-dimensional space in a fixed \smash{1-dimensional} space is $\smash{q^{N-2}}$. Therefore the number of hyperplanes meeting $\langle P \rangle$ in a 1-dimensional space not contained in $P$ is $(q+1-i)\,q^{N-2}$. 
    \item Write $P=\{P_1,\ldots,P_i\}$. Observe that a hyperplane $H_0$ of  $\langle P \rangle$ that distinguishes $P$ is uniquely determined by the intersections $\{H_0\cap (P_1+P_j) \mid j=2,\ldots,i\}$. Indeed, since $\dim(H_0)=i-1$ it follows that  for all $j\in \{2,\ldots,i\}$ we have $\dim(H_0\cap (P_1+P_j)) = 1$, hence $\smash{H_0 = \sum_{j=2}^i H_0\cap (P_1+P_j)}$. For every $j\in \{2,\ldots,i\}$ there are $q-1$ distinct \smash{1-dimensional} subspaces in $(P_1+P_j)$ that are different from $P_1$ and $P_j$, and thus there are $(q-1)^{i-1}$ different hyperplanes $H_0$ of $\langle P \rangle$ not containing any element of~$P$. For each such hyperplane $H_0$ there are $q^{N-i}$ hyperplanes of $X$ whose intersection with $\langle P \rangle$ is exactly $H_0$, and the result follows. \qedhere
\end{enumerate}
\end{proof}

\begin{remark}
Following the notation of the previous result, it is interesting to observe that
\begin{equation} \label{toshow}
(q-1)^i\,q^{N-i} > (q+1-i)(q-1)q^{N-2} \quad \mbox{for $i \ge 3$},
\end{equation}
as we will show shortly.
In other words, a point set with large span is distinguished by more hyperplanes than a point set with small span (for a given cardinality of the point set). 

To see why the inequality in~\eqref{toshow} holds, observe first that it is 
equivalent to
\begin{equation} \label{equ}
    (q-1)^{i-1} > (q+1-i)\,q^{i-2}.
\end{equation}
This is trivially true if $i \ge q+1$, as in that case the RHS of~\eqref{equ} is negative while its LHS is positive. If $3 \le i \le q$ we have
\[ (q-1)^{i-1}=\sum_{k=0}^{i-1} {i-1 \choose k} (-1)^k q^{i-1-k} \]
and therefore \eqref{equ} holds if and only if
\[ \sum_{k=2}^{i-1} {i-1 \choose k} (-1)^k q^{i-1-k}>0. \]
In order to prove the latter inequality, it suffices to show that
\begin{equation}\label{eq:ini-1-k}
{i-1 \choose k} q^{i-1-k}>{i-1 \choose k+1} q^{i-2-k} \quad \mbox{ for $k$ even and $2\leq k\leq i-2$}.
\end{equation}
This is equivalent to  
$q>\frac{i-1-k}{k+1}$,
which is true under our assumptions.
\end{remark}

In our next result we formalize the connection between $\dim(P)$
and the value of $\delta_q(X,k,P)$. More precisely, we study the average value of $\delta_q(X,k,P)$,
as $P$ ranges over all the subsets $P \subseteq \mG_q(X,1)$ having prescribed cardinality and dimension of the span.

\begin{notation} \label{not:rho}
 For integers $2 \le \rho \le N$ and $\rho \le \ell \le (q^\rho-1)/(q-1)$, we let
\[
\hat{\delta}_q(N,k,\ell,\rho):= \frac{\displaystyle\sum_{\substack{P \subseteq \mG_q(X,1) \\ |P|=\ell, \, \dim(P)=\rho}} \delta_q(X,k,P)}{|\{P \subseteq \mG_q(X,1) \mid 
\dim(P)=\rho, \, 
|P|=\ell\}|}\]
denote the average density of the $k$-dimensional subspaces of $X$ that
distinguish a point set~$P$, as $P$ ranges over all point sets of size $\ell$ and rank $\rho$. Clearly, the choice of $X$ is irrelevant.
\end{notation}

We will give a closed formula for the average defined in Notation~\ref{not:rho}.
We start by introducing the following quantity.

\begin{notation}\label{not:lambda}
For integers $2 \le \rho \le N$, $1 \le s \le N$ and $\rho \le \ell \le (q^\rho-1)/(q-1)$, let
\begin{multline*}
    \lambda_q(N,s,\ell,\rho):= \\ \sum_{i=0}^\rho (-1)^{\rho-i} \, q^{\binom{\rho-i}{2}} \, \qbin{N-i}{\rho-i}{q} \sum_{t=0}^s \dbinom{\frac{q^i-q^t}{q-1}}{\ell} \,   \qbin{s}{t}{q} \, \qbin{N-s}{i-t}{q}\, q^{(s-t)(i-t)}.
\end{multline*}
\end{notation}

The following theorem is the main result of this section.

\begin{theorem} \label{mainthavg}
Let $2 \le \rho \le N$ and $\rho \le \ell \le (q^\rho-1)/(q-1)$ be integers. We have
\[
\hat{\delta}_q(N,k,\ell,\rho)= \frac{\lambda_q(N,k,\ell,\rho)}{\lambda_q(N,0,\ell,\rho)}.
\]
\end{theorem}

The proof of Theorem~\ref{mainthavg} relies on the following technical result, which gives the quantity introduced in Notation~\ref{not:lambda} a precise combinatorial significance.

\begin{lemma}\label{lem:lambda}
Let $2 \le \rho \le N$ and $\rho \le \ell \le (q^\rho-1)/(q-1)$ be integers. Let $V \le X$ be a subspace of dimension $1 \le s \le N$. The number of point sets $P \subseteq \mG_q(X,1)$ with $\dim(P)=\rho$,
$|P|=\ell$, and such that~$V$ distinguishes $P$ is $\lambda_q(N,s,\ell,\rho)$.
\end{lemma}
\begin{proof}
We will use M\"obius inversion in the lattice of subspaces of $X$; see \cite[Chapter 3]{stanley2011enumerative} for a general reference. For a subspace $W \le X$, define
\begin{align*}
    f(W) &:= |\{P \subseteq \mG_q(X,1) \mid  \langle P \rangle=W, \,  |P|=\ell \}|, \\
    g(W) &:= \sum_{W' \le W} f(W').
\end{align*}
It follows from the definitions that
\begin{equation} \label{ggg}
    g(W)  = \dbinom{\frac{q^i-q^t}{q-1}}{\ell}, \quad \mbox{if $i=\dim(W)$ and $t=\dim(W \cap V)$}.
\end{equation}
By applying M\"obius inversion, we can compute the desired quantity as
\begin{align*}
    \sum_{\substack{W \le X \\ \dim(W)=\rho}}f(W) &= 
    \sum_{\substack{W \le X \\ \dim(W)=\rho}} \sum_{W' \le W} g(W') \mu(W',W),
\end{align*}
where $\mu$ is the M\"obius function of the lattice of subspaces of $X$; see for instance~\cite[Example~3.10.2]{stanley2011enumerative}. 
Using~\eqref{ggg}, after straightforward computations one gets
\begin{align}\label{interm}
    \sum_{\substack{W \le X \\ \dim(W)=\rho}}f(W) &= 
    \sum_{i=0}^\rho \sum_{t=0}^s \dbinom{\frac{q^i-q^t}{q-1}}{\ell} \, (-1)^{\rho-i} \, q^{\binom{\rho-i}{2}} \, h(i,t),
\end{align}
where
$$h(i,t)=|\{(W,W') \mid W' \le W \le X, \, \dim(W)=\rho, \, \dim(W')=i, \, \dim(W' \cap V)=t\}|.$$
The latter quantity can be computed with the aid of~\cite[Lemma~1]{braun2018q} as
\begin{equation*}
    h(i,t) = \qbin{N-i}{\rho-i}{q} \,  \qbin{s}{t}{q} \, \qbin{N-s}{i-t}{q}\, q^{(s-t)(i-t)}. \qedhere
\end{equation*}
\end{proof}

\begin{proof}[Proof of Theorem~\ref{mainthavg}]
We count in two ways the elements of the set $$\{(P,V) \mid P\subseteq \mG_q(X,1), \, \dim(P)=\rho, \, |P|=\ell, \,  V\in \mG_q(X,k), \, V \mbox{ distinguishes } P \}$$
and use Lemma~\ref{lem:lambda},
obtaining
$$\displaystyle\sum_{\substack{P \subseteq \mG_q(X,1) \\ \dim(P)=\rho, \, |P|=\ell}} \qbin{N}{k}{q}\,\delta_q(X,k,P)=\qbin{N}{k}{q} \, \lambda_q(N,k,\ell,\rho).$$
Dividing both sides of the previous identity by $\qbin{N}{k}{q}\, \lambda_q(N,k,\ell,\rho)$ gives the desired expression for
$\hat{\delta}_q(N,k,\ell,\rho)$.
\end{proof}

Determining the exact connection between $\rho$ and $\hat{\delta}_q(N,k,\ell,\rho)$ seems to be a challenging task in general. We propose a detailed study of this connection in an example, which also reflects the general behaviour we observed in computer experiments.

\begin{example}
Let $(q,N,k)=(2,10,6)$. We want to study the value of $\hat{\delta}_q(N,k,\ell,\rho)$ as a function of $\rho$. We fix an ambient space $X$ of dimension $N$ for convenience (as already stated, the choice of $X$ is irrelevant). We take $\smash{\ell=(q^{N-k+1}-1)/(q-1)}=(2^{10-6+1}-1)/(2-1)=2^5-1$, which is the number of 1-dimensional subspaces of a subspace of $X$ with dimension $N-k+1$.
The values of $\hat{\delta}_2(10,6,2^5-1,\rho)$ are as follows (truncated after four decimal digits):
\medskip
\begin{center}
    \begin{tabular}{|c|c|c|c|c|c|}
    \hline
$\rho=10$ & $\rho=9$ & $\rho=8$ & $\rho=7$ 
& $\rho=6$ & $\rho=5$ \\ 
\hline
0.1352 & 0.1333 & 0.1295 & 0.1211 & 0.1003 & 0 \\
\hline
\end{tabular}
\end{center}
\medskip
The data show that the average density of the spaces distinguishing a point set of a given cardinality decreases with the dimension of the space spanned by the point set. 

It is no surprise that the value corresponding to $\rho=5$ is $0$. Indeed, a point set $P$ of size $\ell=2^{5}-1$ and $\dim(P)=5$ is \textit{necessarily} the set of $1$-dimensional subspaces of a 
$5$-dimensional space, and there is no $6$-dimensional space distinguishing such a point set $P$. Interestingly, lowering the value of $\ell$ by just 1 is enough for point sets $P$ having spaces distinguishing it. Indeed, we have
$$\hat{\delta}_q(10,6,2^5-2,5) > 0.
$$
\end{example}

\subsection{Points, Hyperplanes, and Hamming-Metric Codes}
In this subsection we concentrate on the Critical Problem in the case where $k=N-1$ and $\dim(P)=N$. In other words, we are interested in counting the number of hyperplanes distinguishing a point set spanning the entire ambient space $X$. 

There is an interesting connection between the instance of the Critical Problem we just described and the theory of Hamming-metric codes, which we now illustrate.

\begin{definition}
A (\textbf{Hamming-metric}) \textbf{block code} of \textbf{length} $\ell$ and \textbf{dimension} $N$ is a subspace $C \le \F_q^\ell$ of dimension $N$. Its elements are called \textbf{codewords}. The \textbf{dual} of $C$ is the block code 
$\smash{C^\perp=\{y \in \F_q^\ell \mid y \cdot x^{\top}=0 \mbox{ for all $x \in C$}\}}$. Its dimension is $\ell-N$.

The \textbf{Hamming weight} of a vector $\smash{x \in \F_q^\ell}$ is the integer
$\smash{\wH(x)=|\{i \mid x_i \neq 0\}|}$. For a block code $C$ and an integer $j$, we let $W_j(C)$ denote the number of vectors $x \in C$ with $\wH(x)=j$. The sequence $(W_j(C))_j$ is the \textbf{weight distribution} of $C$.
\end{definition}

\begin{definition}
Let $\smash{P \subseteq \mG_q(X,1)}$ be a point set such that $\dim(P)=N$ and fix an isomorphism~$\smash{X \cong \F_q^N}$. Let $\ell=|P|$ and write $\smash{P=\{P_1,...,P_\ell\}}$. For all $i \in \{1,...,\ell\}$, fix a non-zero vector $\smash{p_i \in \F_q^N}$ that spans~$P_i$.
Finally, let $\smash{G \in \F_q^{N \times \ell}}$
be the matrix whose columns are $p_1,...,p_\ell$.
Since $\dim(P)=N$, we have that $G$ has rank $N$. The row-space of $G$ is a block code of length~$\ell$ and dimension~$N$, which we denote by $C_P$ and call \textbf{associated} to $P$.
\end{definition}

Note that the block code $C_P$ defined above is not unique. It depends on the choice of the isomorphism 
$\smash{X \cong \F_q^N}$, on the order of the $P_i$'s, and on the choice of the $p_i$'s. However, it is not difficult to see (and very well known in coding theory) that the weight distributions of~$C_P$ and~$C_P^\perp$ do not depend on any of these choices.

Let the \textbf{kernel} of  $\smash{x \in \F_q^N}$ be defined as $\smash{\ker(x):=\{y \in \F_q^N \mid y \cdot x^{\top}=0\}}$.
All hyperplanes $\smash{H \le \F_q^N}$ are of the form $H=\ker(x)$ for some non-zero vector $x \in \F_q^\ell$. Moreover,~$x$ is unique up to multiplication by a non-zero field element.
It follows from the definitions that for a non-zero vector $x \in \F_q^N$ the following are equivalent:
\begin{itemize}
    \item $\ker(x)$ is a hyperplane distinguishing $P$;
    \item $x \cdot G$ is a vector of non-zero entries, i.e., of Hamming weight $\ell$.
\end{itemize}
Since $G$ has rank $N$, the map $\F_q^N \to \F_q^\ell$ defined by $x \mapsto x \cdot G$ is injective. Therefore all of this establishes the following result.

\begin{proposition} \label{CW}
Let $\smash{P \subseteq \mG_q(X,1)}$ be a point set with $\dim(P)=N$ and let $C_P$ be a block code associated to $P$. Let $\ell=|P|$. Then
$$\delta_q(X,N-1,P) = \frac{W_\ell(C_P)}{q^N-1}.$$
\end{proposition}
Therefore, the Critical Problem for $k=N-1$ (hyperplanes) is equivalent to the problem of computing the number of codewords of maximum weight in a block code. This connection allows us to solve the Critical Problem in some special instances by using coding theory results.

\begin{proposition} \label{prop:MDS}
Let $\smash{P \subseteq \mG_q(X,1)}$ be an arc in $X$, i.e., a point set in which every $N$ elements span $X$. Let $\ell=|P|$. Then
\begin{equation} \label{f:mds}
\delta_q(X,N-1,P) = \frac{q-1}{q^N-1} \sum_{j=0}^{N-1} (-1)^j \binom{\ell-1}{j} \,q^{N-j-1}.  \end{equation}
\end{proposition}
\begin{proof}
By construction, $C_P$ is a so-called MDS code; see~\cite[Section 7.4]{huffman2010fundamentals}.
Its weight distribution is known and can be found in \cite[Theorem~7.4.1]{huffman2010fundamentals}. In this paper, we find the expression
of \cite[Exercise~392(b)]{huffman2010fundamentals} more helpful in our analysis.
Combining that expression with
Proposition~\ref{CW} gives the desired formula.
\end{proof}

It is a long-standing conjecture that the largest size of an arc in an $N$-dimensional space~$X$ over $\F_q$ is $q+1$ (except possibly when $q$ is even and $N \in \{3,q-1\}$, in which case the maximum value is $q+2$). This is the famous MDS Conjecture. It is therefore natural to compare the density value computed in Proposition~\ref{prop:MDS} for
$\ell=q+1$ and $\hat\delta_q(N,N-1,q+1)$, in the limit as~$q$ tends to infinity.
Indeed, the latter is the average density of the hyperplanes distinguishing a uniformly random point set of size~$q+1$. It easily follows from Proposition~\ref{pp} that
\begin{equation} \label{n1}
\hat\delta_q(N,N-1,q+1) \sim 1/e \quad \mbox{ as $q \to +\infty$.}
\end{equation}
It remains to compute the asymptotics of \eqref{f:mds} for $\ell=q+1$ and $q$ large. 
Using Stirling's approximation we find $$\binom{q}{j} \sim  \frac{q^j}{j!} \quad \mbox{as $q \to +\infty$,}$$
from which we conclude that
\begin{equation} \label{mds:lim}
\frac{q-1}{q^N-1} \sum_{j=0}^{N-1} (-1)^j \binom{\ell-1}{j} \,q^{N-j-1} \, \sim \, \sum_{j=0}^{N-1} \frac{(-1)^j}{j!} \quad \mbox{as $q \to +\infty$.}
\end{equation}

\begin{remark}
We compare~\eqref{n1} and~\eqref{mds:lim}. The quantity on the RHS of \eqref{mds:lim} is extremely close to $1/e$ (for example, for $N=10$ the difference between the two quantities in absolute value is smaller than $10^{-6}$). In particular, 
the previous discussion shows that a uniformly random point set of cardinality $q+1$ behaves like an arc, in the limit as $q \to +\infty$, with respect to the density of hyperplanes distinguishing it.
\end{remark}

It is natural to ask whether arcs maximize the number of distinguishing hyperplanes among all sets of a certain cardinality. The answer to this question is negative in general. In fact, in the next example 
we show how one can explicitly construct point sets having more distinguishing hyperplanes than arcs.

\begin{example}
Let $\ell \le q-1$ and consider a matrix $G$ of the form
$$G= \begin{pmatrix}
 &  &  &  &  &  &  & 0 \\
 &  &  & G' &  &  &  & \vdots \\
  &  &  &  &  &  &  & 0 \\ 
0 & \cdots  &  &  &  & \cdots & 0 & 1 
\end{pmatrix} \in \F_q^{N \times \ell},$$
where the spans of the columns of $G'$ form an arc in $\F_q^{N-1}$. Let $P \subseteq \mG_q(\F_q^N,1)$ be the point set defined by the spans of the columns of $G$. By definition,
$C_P$ is the row-space of $G$. Moreover,
the row-space of $G'$ is an MDS code. Therefore, 
again by~\cite[Exercise~392(b)]{huffman2010fundamentals}, the number of vectors of Hamming weight $\ell-1$ in the row-space of~$G'$ is given by
$$(q-1)\sum_{j=0}^{N-2} (-1)^j \binom{\ell-2}{j} \,  q^{N-j-2}.$$
As a consequence, the number of vectors in $C_P$ of Hamming weight $\ell$ is
$$(q-1) \cdot (q-1) \sum_{j=0}^{N-2} (-1)^j \binom{\ell-2}{j} \, q^{N-j-2}=(q-1)^2 \sum_{j=0}^{N-2} (-1)^j \binom{\ell-2}{j} \, q^{N-j-2}$$
and by Proposition~\ref{CW} we have
\begin{equation} \label{term}\delta_q(\F_q^N,N-1,P) =
\frac{(q-1)^2}{q^N-1} \,  \sum_{j=0}^{N-2} (-1)^j \binom{\ell-2}{j} \,
q^{N-j-2}.
\end{equation}
Tedious computations show that 
the difference between~\eqref{term} and the expression in Proposition~\ref{prop:MDS} is 
$$\frac{q-1}{q^N-1}(-1)^N \binom{\ell-2}{N-1},$$
which is positive whenever
$N$ is even and $\ell \ge N+1$.
This shows that, under those assumptions, the point set $P$ constructed in this example has more distinguishing hyperplanes than an arc of the same cardinality.
\end{example}

\medskip
\section{Other Density Functions of Rank-Metric Codes}
\label{sec:special}
In this section we investigate
the density functions of some
special families of rank-metric codes. The section is overall divided into four subsections, each of which concentrates on particular code parameters, or on constraints imposed on the matrices (symmetric, alternating, and Hermitian).
More details about the results can be found at the beginning of each subsection.

\subsection{Square Codes and Tensors}
We compute the exact value of the density 
function of 2-dimensional, full-rank, $n \times n$ codes. We then observe that 
its asymptotics for both $q$ and $n$ large coincides with the asymptotic density of $2\times n$ full-rank MRD codes. Finally, we explain this analogy by connecting rank-metric codes having certain parameters with 3-dimensional \textit{tensors}.

In the proof of Theorem~\ref{thm:exact2dim} below we will need the asymptotics of the number of spectrum-free matrices in $\F_q^{n \times n}$, $s_q(n)$, both for $q \to+\infty$ and for $n \to+\infty$; see~\eqref{eq:spect} for the definition of~$s_q(n)$.
These were obtained in~\cite[Theorem VII.1]{antrobus2019maximal} and read as follows:
\begin{align} \label{eq:specti}
  \lim_{q \to +\infty} \frac{s_q(n)}{q^{n^2}} = \sum_{i=0}^n \frac{(-1)^i}{i!}, \qquad \lim_{n \to +\infty} \frac{s_q(n)}{q^{n^2}} = \displaystyle\prod_{i=1}^{\infty} \left(1-\frac{1}{q^i}\right)^{q}.
\end{align}


\begin{remark} \label{rem:spectri}
We have $s_q(n)=|\{M \in \F_q^{n \times n} \mid \det(\lambda I + M) \ne 0 \text{ for all }\lambda \in \F_q\}|$,
where $I \in \F_q^{n \times n}$ denotes the identity matrix. Moreover, it is easy to check that we even have
\begin{align} \label{eq:spec}
    s_q(n) = |\{M \in \F_q^{n \times n} \mid \det(\lambda N + M) \ne 0 \text{ for all }  \lambda \in \F_q\}|
\end{align}
for any invertible matrix $N \in \F_q^{n \times n}$.
\end{remark}

We can now compute the exact density of 2-dimensional $n \times n$ full-rank codes and its asymptotics as $q$ and $n$ grow.

\begin{theorem} \label{thm:exact2dim}
We have
$$\delta^\rk_q(n\times n, 2, n) = \frac{s_q(n)\prod_{i=0}^{n-1}(q^n-q^i)}{(q^{n^2}-1)(q^{n^2}-q)}.$$
In particular, 
\begin{align*}
    \lim_{q \to +\infty} \, \delta^\rk_q(n \times n, 2, n) = \sum_{i=0}^n \frac{(-1)^i}{i!},  \qquad
    \lim_{n \to +\infty} \, \delta^\rk_q(n \times n, 2, n) = \displaystyle\prod_{i=1}^{\infty} \left(1-\frac{1}{q^i}\right)^{q+1}.
    \end{align*}
\end{theorem}
\begin{proof} 
Let $B=B_q^\rk(n \times n,1)$ and consider the following set:
\begin{align*}
    \mathfrak{A} := \{(\mC,\mD) \mid \mD \le \mC \le \F_q^{n \times n}, \, \dim(\mC)=2, \, \dim(\mD)=1, \, \mC \text{ distinguishes } B \}.
\end{align*}
On the one hand we have
\begin{align*}
    |\mathfrak{A}| = |\{\mC \le \F_q^{n \times n} \mid \dim(\mC)=2, \, \mC \text{ distinguishes } B\}| \cdot \qbin{2}{1}{q}.
\end{align*}
On the other hand,
\begin{align*}
\allowdisplaybreaks
|\mathfrak{A}| &= \sum_{\substack{\mD \le \F_q^{n \times n} \\ \dim(\mD)=1}} |\{\mC \le \F_q^{n \times n} \mid \dim(\mC)=2, \, \mD \le \mC, \, \mC \text{ distinguishes } B\}|.
\end{align*}
Let
\begin{align*}
\mathfrak{B} := \{(\mC,N) \mid \mC \le \F_q^{n \times n}, \, \dim(\mC)=2, \, N \in \mC \setminus \{0\}, \, \mC \text{ distinguishes } B \}
\end{align*}
and define $\varphi: \mathfrak{B} \longrightarrow \mathfrak{A}$ by $\varphi(\mC,N):=(\mC,\langle N \rangle)$
for all $(\mC,N) \in \mathfrak{B}$.
It is easy to check that~$\varphi$ is well-defined and surjective. Moreover, 
 $|\varphi^{-1}(\mC,\mD)| = q-1$ for all $(\mC,\mD) \in \mathfrak{A}$, because $|\mD\setminus\{0\}|= q-1$ and every $N \in \mD\setminus\{0\}$ satisfies $\varphi(\mC,N)=(\mC,\mD)$. All of this yields
\begin{multline} \label{latter}
\sum_{\substack{\mD \le \F_q^{n \times n} \\ \dim(\mD)=1}} |\{\mC \le \F_q^{n \times n} \mid \dim(\mC)=2, \, \mD \le \mC, \, \mC \text{ distinguishes } B\}| \\
= \frac{1}{q-1}\, \sum_{\substack{N \in \F_q^{n \times n} \\ \det(N)\ne 0}} |\{\mC \le \F_q^{n \times n} \mid \dim(\mC)=2, \, N \in \mC, \, \mC \text{ distinguishes } B\}|.
\end{multline}
Now fix an invertible matrix $N \in \F_q^{n \times n}$ and define the sets
\begin{align*}
\mathfrak{C} &= \{M \in \F_q^{n \times n} \mid \det(\lambda N + M) \ne 0 \text{ for all }\lambda \in \F_q\}, \\
\mathfrak{D} &= \{\mC \le \F_q^{n \times n} \mid \dim(\mC)=2, \,\langle N \rangle \le \mC, \, \mC \text{ distinguishes } B\}.
\end{align*}
Moreover, let 
$\psi: \mathfrak{C} \to \mathfrak{D}$
be the map
 defined by
$\psi: M \mapsto \langle M,N \rangle$
for all $M \in \mathfrak{C}$.
One can check that $\psi$ is well-defined and surjective. Moreover, for all $\mC \in \mathfrak{D}$ we have $|\psi^{-1}(\mC)|=q^2-q$. We have $|\mC\backslash \langle N \rangle|=q^2-q$, and since $\mC$ distinguishes $B$, each $M \in \mC\backslash \langle N \rangle$ is in $\mathfrak{C}$. Therefore, for all invertible matrices $N \in \F_q^{n \times n}$ we have
\begin{multline} \label{pf:exact2dim1}
|\{\mC \le \F_q^{n \times n} \mid \dim(\mC)=2, \, N \in \mC, \, \mC \text{ distinguishes } B\}| \\
= \frac{|\{M \in \F_q^{n \times n} \mid \det(\lambda N + M) \ne 0 \text{ for all }\lambda \in \F_q\}|}{q^2-q}.
\end{multline}

To conclude the proof, we combine Remark~\ref{rem:spectri} with Equations~\eqref{latter} and~\eqref{pf:exact2dim1}, obtaining
\begin{align*}
    |\mathfrak{A}|
    = &\frac{1}{q-1}\sum_{\substack{N \in \F_q^{n \times n} \\ \det(N)\ne 0}} \frac{s_q(n)}{q^2-q} = \frac{s_q(n)\, \prod_{i=0}^{n-1}(q^n-q^i)}{q(q-1)^2}.
\end{align*}
Using the following identity (see e.g.~\cite[Section 3]{andrews1998theory})
\begin{align} \label{binomid}
    \qbin{a}{b}{q} \, \qbin{b}{r}{q} = \qbin{a}{r}{q} \,  \qbin{a-r}{a-b}{q},
\end{align}
we compute the desired density as
\begin{align*}
\delta^\rk_q(n \times n, 2, n) = \frac{|\mathfrak{A}|}{\qbin{2}{1}{q} \, \qbin{n^2}{2}{q}} &= \frac{s_q(n) \cdot \prod_{i=0}^{n-1}(q^n-q^i)}{q(q-1)^2 \, \qbin{n^2}{1}{q} \,  \qbin{n^2-1}{1}{q}} \\
&= \frac{s_q(n)\prod_{i=0}^{n-1}(q^n-q^i)}{q(q^{n^2}-1)(q^{n^2-1}-1)}.
\end{align*}

Finally, the two asymptotic estimates in the theorem are straightforward consequences of~\eqref{eq:specti} and the fact that
\begin{equation*}
\frac{\prod_{i=0}^{n-1}(q^n-q^i)}{q^{n^2}} = \frac{q^{n^2}\prod_{i=1}^{n} \left(1-q^{-i}\right)}{q^{n^2}} \sim \displaystyle\prod_{i=1}^{\infty} \left(1-\frac{1}{q^{i}}\right) \quad \textnormal{ as $n \to +\infty$}. \qedhere
\end{equation*}
\end{proof}

It is interesting to observe that the asymptotic estimates of $\delta^\rk_q(n \times n, 2, n)$ in Theorem~\ref{thm:exact2dim} are the same as the asymptotic estimates of $\delta^\rk_q(2 \times n, n, 2)$ in~\cite{antrobus2019maximal} both as $q \to +\infty$ and $n \to +\infty$; see Theorem~\ref{hejar}.
The next result shows that this fact is not a coincidence. Indeed, 
the two density functions 
can be related by considering {\it tensors}. 

\begin{theorem} \label{thm:tens}
Let $1 \le r \le n$. We have
\[
\frac{\delta^\rk_q(r \times n, n, r)}{\delta^\rk_q(n \times n, r, n)} = \frac{|\mathrm{GL}_r(q)|}{|\GL|}\frac{\qbin{n^2}{r}{q}}{\qbin{rn}{n}{q}}.
\]
Moreover, we have
\[
\lim_{q \to+\infty} \frac{\delta^\rk_q(r \times n, n, r)}{\delta^\rk_q(n \times n, r, n)} =  \lim_{n \to+\infty} \frac{\delta^\rk_q(r \times n, n, r)}{\delta^\rk_q(n \times n, r, n)} = 1.
\]
\end{theorem}
\begin{proof}
Following the proof of \cite[Theorem 4]{sheekey2019binary}, an $n$-dimensional subspace of $\smash{\fq^{r \times n}}$ defines $\smash{|\GL|}$ different $r\times n\times n$ tensors; one for each ordered basis of the subspace. Similarly, a $r$-dimensional subspace of $\fq^{n \times n}$ defines $|\mathrm{GL}_r(q)|$ different $r\times n\times n$ tensors. As proved in \cite[Theorem 4]{sheekey2019binary}, the set of tensors obtained from $n$-dimensional subspaces of $\smash{\fq^{r \times n}}$ with minimum rank-distance $r$ coincides with the set of tensors obtained from $r$-dimensional subspaces of $\smash{\fq^{n \times n}}$ with minimum rank-distance $n$. Counting the number of such tensors in two ways gives the identity in the statement.

The first limit immediately follows from the asymptotic estimate of the $q$-binomial coefficient and the fact that $\smash{|\mbox{GL}_a(q)| \sim q^{a^2}}$ as $q \to +\infty$ for all positive integers $a$.
To compute the second limit, note that
\begin{align} \label{eq:gln}
    \textnormal{GL}_a(q) = \prod_{i=0}^{a-1}(q^a-q^i) = q^{a^2}\prod_{i=1}^{a} \left(1-\frac{1}{q^i}\right) = \frac{q^{a^2}}{\pi(q,a)}.
\end{align}
By~\eqref{eq:pias} and~\eqref{eq:pias2} we have
\begin{align*}
     {|\mathrm{GL}_r(q)| \ \qbin{n^2}{r}{q}} &\sim \frac{\pi(q,r) q^{r(n^2-r)+r^2}}{\pi(q,r)} = q^{rn^2} \quad \mbox{as $n\to +\infty$}, \\ 
    {|\GL|} \ \qbin{rn}{n}{q} &\sim \frac{\pi(q)q^{n(rn-n)+n^2}}{\pi(q)} = q^{rn^2}  \quad \mbox{as $n\to +\infty$},
\end{align*}
which together establish the second limit.
\end{proof}

\subsection{Symmetric Codes} \label{subsec:sym}
In this subsection we give bounds on the number of symmetric codes with a focus on the asymptotic behaviour of their density as the field size $q$ tends to infinity. We also discuss the connection of the obtained results with the theory of semifields. 

\begin{definition}
A rank-metric code $\mC \le \F_q^{n \times n}$ is called \textbf{symmetric} if all matrices $M \in \mC$ are symmetric. We denote the $\F_q$-space of $n \times n$ symmetric matrices over $\F_q$ by $\mbox{Sym}_n(q)$.
\end{definition}

Clearly, $\mbox{Sym}_n(q)$ is a vector space of dimension $n(n+1)/2$ over $\F_q$. Furthermore the following holds. 

\begin{theorem}[\text{\cite[Theorem 3.3]{schmidt2015symmetric}}] \label{thm:kai}
Let $\mC \le \mbox{Sym}_n(q)$ be a symmetric rank-metric code of minimum distance $d$. We have
\begin{align*}
    \dim(\mC) \le \begin{cases}
    n(n-d+2)/2 \quad &\textnormal{if $n-d$ is even,} \\
    (n+1)(n-d+1)/2 \quad &\textnormal{if $n-d$ is odd.}
    \end{cases}
\end{align*}
\end{theorem}

In~\cite[Section 4]{schmidt2015symmetric},
and more recently in \cite{longobardi2020automorphism,zhou2020equivalence}, constructions of codes whose dimensions meet the bounds of Theorem~\ref{thm:kai} were provided, showing in particular that the bounds of Theorem~\ref{thm:kai} are sharp. We call symmetric rank-metric codes attaining these bounds \textbf{symmetric MRD} codes.

\begin{remark}\label{rem:corrMRDcommsem}
Linear full-rank symmetric MRD codes correspond to \emph{commutative semifields} (also called symplectic semifields) with center containing $\fq$. However, as also noted in \cite[Section 7]{sheekey201913}, such a correspondence is not direct since if $(\fqn,+,\star)$ is a commutative semifield, then the associated MRD code  is not necessarily contained in $ \mbox{Sym}_n(q)$. Nevertheless, one can apply some operations (namely the transposition, see \cite[Section 1.4]{lavrauw2011finite}) on $(\fqn,+,\star)$ in such a way that the associated MRD code is contained in $\mbox{Sym}_n(q)$.
\end{remark}

In the rest of this section we will repeatedly use~\cite[Theorem 4.2]{gruica2020common}, which helps with understanding the asymptotic behaviour of codes with minimum distance bounded from below. In order to apply this result, we need the asymptotic size of the ball in $\mbox{Sym}_n(q)$ of a given radius.

\begin{lemma}[\text{\cite[Theorem 3]{carlitz1954sym}}] \label{lem:rkssym}
Let $0 \le i \le n$ be an integer. We have 
\begin{align*}
    |\{M \in \mbox{Sym}_n(q) \mid \rk(M)=i\}| = \prod_{s=1}^{\lfloor i/2 \rfloor} \frac{q^{2s}}{q^{2s}-1} \, \prod_{s=0}^{i-1}\left( q^{n-s}-1 \right).
\end{align*}
\end{lemma}

Clearly, summing the formula in Lemma~\ref{lem:rkssym} over the numbers between $0$ and $r$ gives the ball of radius $r$ in $\mbox{Sym}_n(q)$.

Similarly to Notation~\ref{not:density}, we denote by $\delta_{q}^\Sym(n \times n, k, d)$ the proportion of $k$-dimensional codes in $\mbox{Sym}_n(q)$ of minimum distance at least $d$ within the set of $k$-dimensional codes in~$\mbox{Sym}_n(q)$. 

\begin{corollary} \label{cor:symasy}
Let $0 \le r \le n$ be an integer. We have 
\begin{align*}
    |\{M \in \mbox{Sym}_n(q) \mid \rk(M)\le r\}|  \sim q^{nr-r(r-1)/2} \quad \textnormal{as $q \to +\infty$.} 
\end{align*}
\end{corollary}
\begin{proof}
We have
\begin{align*}
    \prod_{s=1}^{\lfloor i/2 \rfloor} \frac{q^{2s}}{q^{2s}-1} \prod_{s=0}^{i-1}\left( q^{n-s}-1 \right) \sim q^{ni-i(i-1)/2} \quad \textnormal{ as $q \to +\infty$.}
\end{align*}
The map $i \mapsto ni-i(i-1)/2$ attains its maximum at $i=r$ over the set $\{0, \dots,r\}$ and the maximum is $nr-r(r-1)/2$. 
\end{proof}

By combining Corollary~\ref{cor:symasy} with~\cite[Theorem 4.2]{gruica2020common} we obtain the following result.

\begin{theorem} \label{thm:qasysym}
Let
\begin{align*}
    1 \le k \le \begin{cases}
    n(n-d+2)/2 \quad &\textnormal{if $n-d$ is even,} \\
    (n+1)(n-d+1)/2 \quad &\textnormal{if $n-d$ is odd.}
    \end{cases}
\end{align*}
We have
\begin{align*}
    \delta_q^\Sym(n \times n, k, d) \in
    O\left(q^{n(n+1)/2-k+1-n(d-1)+(d-1)(d-2)/2}\right) \quad \textnormal{ as $q \to +\infty$.}
\end{align*}
Moreover, 
\begin{align*}
\lim_{q \to+\infty}\delta^\Sym_q(n \times n,k,d)=
\begin{cases}
1 & \mbox{if $k < \frac{n(n+1)}{2}+1-n(d-1)+\frac{(d-1)(d-2)}{2}$}, \\ 
0 & \mbox{if $k > \frac{n(n+1)}{2}+1-n(d-1)+\frac{(d-1)(d-2)}{2}$}.
\end{cases}
\end{align*}
In particular,
\begin{align*}
    \delta_q^{\textnormal{sym}}(n \times n, n, n) \in O\left(q^{-n+2}\right) \quad \textnormal{ as $q \to +\infty$.}
\end{align*}
\end{theorem}

\begin{remark}
In \cite{kw2004symplectic}, as noted in Remark \ref{rem:kantor}, it was shown that for $q$ even there are a large number of isotopy classes of semifields of order $q^n$. In fact, the semifields constructed in order to prove this result are all commutative, and thus correspond to $\fq$-linear full-rank symmetric MRD codes, giving a lower bound for the density. The recent paper \cite{golkol2021sym} constructs large families of commutative semifields in the case where $q$ is odd. However, the interesting growth of these constructions is for fixed $q$ and increasing~$n$, whereas our results above addresses fixed~$n$ and increasing $q$. There remains a wide gap between the sparseness results above and the lower bounds from these constructions.
\end{remark}

\subsection{Alternating Codes}
In this subsection we study the asymptotic behaviour of the density function of alternating rank-metric codes as $q \to +\infty$. We start by giving the needed definitions and formulas.

\begin{definition}
A matrix $\smash{M \in \mat}$ is \textbf{alternating}, or \textbf{skew-symmetric}, if for all $1 \le i,j \le n$ we have $\smash{M_{ij}=-M_{ji}}$ and $M_{ii} = 0$.
A rank-metric code $\smash{\mC \le \F_q^{n \times n}}$ is called \textbf{alternating} if every $M \in \mC$ is alternating.
 We denote the $\F_q$-space of $n \times n$ alternating matrices over $\F_q$ by~$\mbox{Alt}_n(q)$.
\end{definition}

Note that $\mbox{Alt}_n(q)$ is a vector space of dimension ${n(n-1)}/{2}$ over $\F_q$. Furthermore the following holds. 

\begin{theorem}[\text{\cite[Theorem 4]{delsarte1975alternating}}] \label{thm:altdim}
Let $\C\leq \mathrm{Alt}_n(q)$ be an alternating rank-metric code of minimum distance $2e$ and let $t=\lfloor n/2 \rfloor$. We have
\[ \dim(\C)\leq \frac{n(n-1)}{2t}(t-e+1). \]
\end{theorem}
We call the alternating codes attaining the bound of Theorem~\ref{thm:altdim} \textbf{alternating MRD} codes. In \cite{delsarte1975alternating}, alternating MRD codes were shown to exist for $n$ odd and any $q$, or for $n$ even and $q$ even.

Analogously to Subsection~\ref{subsec:sym} we denote by $\delta_{q}^\Alt(n \times n, k, d)$ the proportion of $k$-dimensional codes in $\mbox{Alt}_n(q)$ of minimum distance at least $d$ within the set of $k$-dimensional codes in $\mbox{Alt}_n(q)$. Note that alternating matrices necessarily have even rank (see e.g.~\cite[Section 10.3]{hoffmann1971linear}) and thus it only makes sense to consider minimum distance $d=2e$ for some non-negative integer $e$.
We have the following formula for the number of alternating matrices of a given rank. 

\begin{lemma}[\text{\cite[Proposition 62]{ravagnani2018duality}}] \label{lem:altnum}
Let $0 \le i \le n$ be an integer. We have 
\begin{align*}
    |\{M \in \mbox{Alt}_n(q) \mid \rk(M)=i\}| = \qbin{n}{i}{q} \sum_{s=0}^i (-1)^{i-s}q^{\binom{s}{2}+\binom{i-s}{2}} \qbin{i}{s}{q}.
\end{align*}
\end{lemma}

It is easy to see that if $i$ is odd, then the formula given in Lemma~\ref{lem:altnum} is equal to 0.

\begin{corollary} \label{cor:altasy}
Let $0 \le r \le n$ be an integer. We have 
\begin{align*}
    |\{M \in \mbox{Alt}_n(q) \mid \rk(M)\le r\}| \sim \begin{cases}
    q^{rn-r(r+1)/2} \quad &\textnormal{ if $r$ is even,} \\
    q^{(r-1)n-(r-1)r/2} \quad &\textnormal{ if $r$ is odd,}
    \end{cases} \quad \textnormal{ as $q \to +\infty$.}
\end{align*}
\end{corollary}
\begin{proof}
Let $0 \le i \le r$ be an integer. For all $s \in \{0, \dots, n\}$ we have
\begin{align*}
    (-1)^{i-s}q^{\binom{s}{2}+\binom{i-s}{2}} \qbin{i}{s}{q} \sim (-1)^{i-s} q^{i(i-1)/2} \quad \textnormal{as $q \to +\infty$.}
\end{align*}
In particular, 
\begin{align*}
    \qbin{n}{i}{q} \sum_{s=0}^i (-1)^{i-s}q^{\binom{s}{2}+\binom{i-s}{2}} \qbin{i}{s}{q} \sim \begin{cases}
    q^{i(n-i)+i(i-1)/2}  &\textnormal{ if $i$ is even,} \\
    0  &\textnormal{ if $i$ is odd,}
    \end{cases} \quad \textnormal{ as $q \to +\infty$.}
\end{align*}
Since the map $i \mapsto i(n-i)+i(i-1)/2$ for $i \in \{0, \dots, r\}$ attains its maximum at $i=r$, the desired result follows.
\end{proof}

By Corollary~\ref{cor:altasy} and~\cite[Theorem 4.2]{gruica2020common} we have the following analogue of Theorem~\ref{thm:qasysym} for the asymptotic density of alternating rank-metric codes as $q \to +\infty$. 

\begin{theorem}
\label{thm:qasysyma}
Let $2 \le d \le n$ be an even integer with $d=2e$ and let \smash{$1 \le k \le \frac{n(n-1)}{2t}(t-e+1)$} (where $t$ is as in Theorem~\ref{thm:altdim}). We have the following asymptotic estimate for the density of alternating rank-metric codes.
\begin{align*}
    \delta_q^\Alt(n \times n, k, d) \in
    O\left(q^{{n(n-1)}/{2}-k+1-(d-2)n+{(d-1)(d-2)}/{2}}\right) \quad \textnormal{ as $q \to +\infty$.}
\end{align*}
Moreover, we have
\begin{align*}
\lim_{q \to+\infty}\delta^\Alt_q(n \times n,k,d)=
\begin{cases}
1 & \mbox{if $k < \frac{n(n-1)}{2}+1-(d-2)n+\frac{(d-1)(d-2)}{2}$}, \\ 
0 & \mbox{if $k > \frac{n(n-1)}{2}+1-(d-2)n+\frac{(d-1)(d-2)}{2}$}.
\end{cases}
\end{align*}
\end{theorem}

One can check that Theorem~\ref{thm:qasysyma} gives the sparseness of alternating MRD codes in $\mbox{Alt}_n(q)$ for any minimum distance $d \ge 4$. Note that minimum distance 2 gives the trivial alternating MRD code $\mC=\mbox{Alt}_n(q)$.

\subsection{Hermitian Codes}

Similarly to Subsection~\ref{subsec:sym}, we provide bounds for the number of Hermitian rank-metric codes and discuss their asymptotic behaviour as $q \to +\infty$. In this subsection we always work over a finite field extension~$\F_{q^2}$ of $q^2$ elements, where $q$ is a prime power.

\begin{definition}
For $x \in \F_{q^2}$ consider the conjugation map defined by $x \mapsto x^q$.
A matrix $\smash{M \in \F_{q^2}^{n \times n}}$ is called \textbf{Hermitian} if $M=M^*$ where $M^*$ is obtained by the transposition of~$M$ and by applying the conjugation map to every entry of $M$.
We call an $\F_q$-linear code $\smash{\mC \le \F_{q^2}^{n \times n}}$ a \textbf{Hermitian code} if every $M \in \mC$ is Hermitian. We denote the $\F_q$-space of Hermitian matrices in~$\smash{\F_{q^2}^{n \times n}}$ by~$\smash{\Her_{n}(q^2)}$. 
\end{definition}


Note that $\Her_{n}(q^2)$ is a vector space of dimension $n^2$ over $\F_q$. Even though Hermitian codes consist of matrices with entries from the field extension $\F_{q^2}$, the linearity requirement is over the small field $\F_q$ (and we still emphasize it with the symbol ``$\le$'').

We have the following upper bound on the dimension a Hermitian rank-metric code with a given minimum distance.

\begin{theorem}[\text{\cite[Theorem 1]{schmidt2018hermitian}}] \label{thm:hermit}
Let $\mC \le \Her_{n}(q^2)$ be a Hermitian rank-metric code. We have
\begin{align*}
    \dim_{\F_q}(\mC) \le n(n-d+1).
\end{align*}
\end{theorem}

As before we call Hermitian rank-metric codes attaining the bound of Theorem~\ref{thm:hermit} \textbf{Hermitian MRD} codes. Note that the existence of Hermitian MRD codes is known except for $n$ and $d$ both even and $4\leq d \leq n-2$; see~\cite{schmidt2018hermitian,trombetti2021maximum} for details. We denote the density function of Hermitian rank-metric codes of dimension $k$ and minimum distance bounded from below by $d$ by $\delta^{\Her}_q(n \times n, k, d)$ (analogously to Notation~\ref{not:density}). 

\begin{lemma}[\text{\cite[Theorem 3]{carlitz1955hermit}}] \label{lem:hermitranks}
Let $0 \le i \le n$ be an integer.
We have
\begin{align*}
    |\{M \in \Her_{n}(q^2) \mid \rk(M) =i \}| = \qbin{n}{i}{q^2}\,q^{\frac{i(i-1)}{2}}\prod_{j=1}^i(q^j-(-1)^j).
\end{align*}
\end{lemma}

The asymptotic estimate of the Hermitian ball of radius $r$ as $q \to +\infty$ is as follows.

\begin{corollary} \label{cor:hermitasy}
Let $0 \le r \le n$ be an integer. We have
\begin{align*}
    |\{M \in \Her_{n}(q^2) \mid \rk(M) \le r \}| \sim  q^{r(2n-r)} \quad &\textnormal{as $q \to +\infty$.}
\end{align*}
\end{corollary}
\begin{proof}
It is not hard to see that for all $0 \le i \le r$ we have
\begin{align*}
    \qbin{n}{i}{q^2}\,q^{\frac{i(i-1)}{2}}\prod_{j=1}^i(q^j-(-1)^j) \sim  q^{2i(n-i)+\frac{i(i-1)}{2}+\frac{i(i+1)}{2}} = q^{2i(n-i)+i^2} \quad \textnormal{as $q \to +\infty$.}
\end{align*}
The dominant term in $q$ is attained for $i=r$ over all $i \in \{0, \dots,r\}$ and its value is $q^{r(2n-r)}$.
\end{proof}

By Corollary~\ref{cor:hermitasy} and~\cite[Theorem 4.2]{gruica2020common} we have the following analogue of Theorem~\ref{thm:qasysym} for the asymptotic density Hermitian rank-metric codes as $q \to +\infty$.

\begin{theorem}
\label{thm:qasysymh}
Let $1\le k \le n(n-d+1)$ be an integer. We have the following asymptotic estimate for the density of Hermitian rank-metric codes:
\begin{align*}
    \delta_q^{\Her}(n \times n, k, d) \in
    O\left(q^{n^2-k+1-(d-1)(2n+d-1)}\right) \quad \textnormal{ as $q \to +\infty$.}
\end{align*}
Moreover, we have
\begin{align*}
\lim_{q \to+\infty}\delta^{\Her}_q(n \times n,k,d)=
\begin{cases}
1 & \mbox{if $k < n^2+1-(d-1)(2n+d-1)$}, \\ 
0 & \mbox{if $k > n^2+1-(d-1)(2n+d-1)$}.
\end{cases}
\end{align*}
\end{theorem}

Theorem~\ref{thm:qasysymh} shows that Hermitian MRD codes are sparse. More precisely, it gives the following asymptotic estimate on the density function of Hermitian MRD codes:
\begin{align*}
    \delta_q^{\Her}(n \times n, n(n-d+1), d) \in O\left(q^{-(d-1)(n-d+1)+1}\right) \quad \textnormal{ as $q \to +\infty$}.
\end{align*}

\begin{remark}
It is interesting to observe that the asymptotic bound on the density function of Hermitian MRD codes is the same as the one for classical MRD codes (see Theorem~\ref{thm:mrdasybound}), even though the existence for Hermitian MRD codes is not known for all parameter sets (and it is well known that MRD codes exist for all parameters). Therefore the upper bound in~\cite{gruica2020common} does not give an indication for how hard it is to show the existence of certain codes in general.
\end{remark}

\bigskip
\bibliographystyle{amsplain}
\bibliography{ourbib}

\end{document}